\documentclass[a4paper,10pt]{article}
\usepackage[utf8]{inputenc}
\usepackage[T1]{fontenc}
\usepackage{lmodern}
\usepackage{a4wide}
\usepackage{geometry}
\usepackage{graphicx}
\usepackage{amsmath, amsfonts, amsthm,amssymb,here,dsfont, hyperref}
\usepackage{color}
\usepackage{version}
\usepackage{todonotes}
\usepackage[lofdepth,lotdepth]{subfig}
\usepackage{wrapfig}
\usepackage{fullpage}
\usepackage{xcolor}
\usepackage{multirow}
\usepackage{caption}

\newcommand{\one}{\mathds{1}}
\newcommand{\argmin}[1]{\underset{#1}{\operatorname{arg}\!\operatorname{min}}\;}

\newtheorem{theorem}{Theorem}
\newtheorem{proposition}{Proposition}

\newtheorem*{assumption*}{Assumption}
\newtheorem{assumption}{Assumption}
\newtheorem{lemma}{Lemma}

\def\argmin{\mathop{\rm argmin}}

\newcommand{\sgn}{{\rm sgn}}

\newcommand{\wrt}{{\em w.r.t.~}}
\newcommand{\iid}{{\rm i.i.d.~}}

\newcommand{\bbR}{\mathbb{R}}
\newcommand{\bX}{\mathbf{X}}
\newcommand{\bx}{\mathbf{x}}

\begin{document}

\title{Prediction intervals with controlled length in the \\ heteroscedastic Gaussian regression}
 \author{Christophe Denis,
 Mohamed Hebiri, 
and Ahmed Zaoui \\
\small{LAMA, UMR-CNRS 8050,}\\
\small{Universit\'e Gustave Eiffel}\\}
 \date{}
 
\maketitle

\begin{abstract}

We tackle the problem of building a prediction interval in heteroscedastic Gaussian regression. We focus on prediction intervals with constrained expected length in order to guarantee interpretability of the output. 
In this framework, we derive a closed form expression of the optimal prediction interval that allows for the development a data-driven prediction interval based on plug-in. The construction of the proposed algorithm is based on two samples, one labeled and another unlabeled. Under mild conditions, we show that our procedure is asymptotically as good as the optimal prediction interval both in terms of expected length and error rate. In particular, the control of the expected length is distribution-free. We also derive rates of convergence under smoothness and the Tsybakov noise conditions.  We conduct a numerical analysis that exhibits the good performance of our method. It also indicates that even with a few amount of unlabeled data, our method is very effective in enforcing the length constraint.
\end{abstract}
\section{Introduction}
\label{sec:intro}

Prediction is one of the main goals in supervised learning, it consists in building, given historical data, a candidate output for a new observation.
One common practice thereafter is to carry out inference on the output and then to ask for confidence in the predicted value, therefore, \emph{prediction interval (PI)} appears as appropriate tools to handle this problem in the regression setting. A typical application is the prediction in the linear regression case when the data are assumed Gaussian with common variance. In this context, the notion of PI is well studied and well understood both from practice and theory.

However, in the general case, inference as a post-processing step may produce irrelevant conclusions due to the stochastic nature of the data-driven prediction procedure (see for instance~\cite{Berk_inference_13}). Therefore, in order to guarantee the theoretical validity of the prediction intervals, it is suitable to process at once both aspects of the problem, that is, one might design a data-driven procedure directly devoted to the \emph{prediction interval} purpose.


In a classical setting of PI, one often asks for a pre-specified level of confidence for the predicted range of values (says $95\%$ or $99\%$ according to the problem).
This is for instance the approach that is considered in the \emph{conformal prediction} literature~\cite{Vovk_Nouretdinov_Gammerman09,Vovk_Gammerman_Shafer05, Lei_Wasserman13, Lei_G'Sell_Rinaldo_Tibshirani_Wasserman18}. 
However, this strategy may suffer from interpretability issues for problems where prediction task is difficult or when classical assumptions on the noise are not satisfied. Specifically, 
for relatively restrictive values of the confidence level, the resulting output might be so large that it becomes useless.


In contrast, our purpose is to produce for future observation a prediction interval with a pre-determined expected length. This framework is completely different from the previous one since it does not ensure any coverage guarantee but rather ensures the interpretability of the predicted output. Indeed, since the length of the output interval is controlled, we do not expect for a given input instance $\bx \in \bbR^d$ a too large set of candidate values.

Generally speaking, the range of values that we would output with PI has no reason to be an interval. However, in a Gaussian model, this range of values indeed forms an interval (or a union of it). In this paper, we investigate the problem of PI under  expected length constraint in the Gaussian regression setup. We aim at providing a general device that outputs a PI for a new feature. 
Our procedure relies on the plug-in principle and we propose in the present contribution a statistical analysis of it in this setting.



\paragraph*{Main contributions. } Denote by $\Gamma:\bbR^d \to \mathcal{P}(\bbR)$ a given prediction set, where $\mathcal{P}(\bbR)$ is the set of subsets of $\bbR$. One of the main inputs of the present work is the introduction of a novel framework for PI in the regression setting, taking sides of controlling the expected size $\mathbb{E}\left[L(\Gamma(\bX))\right]$ of the output predictor $\Gamma$ while minimizing its error rate $\mathbb{P}\left(Y\notin \Gamma(\bX)\right)$, where $L(\Gamma(\bX)) = \int_\bbR \one_{y\in \Gamma(\bX) } dy $ stands for the Lebesgue measure of $\Gamma$. We derive the optimal rule for this problem which is defined as
\begin{equation*}
\Gamma_{\ell}^{*}\in \argmin_{\Gamma: \mathbb{E}\left[L(\Gamma(\bX))\right]\leq \ell}\mathbb{P}\left(Y\notin \Gamma(\bX)\right)\enspace,
\end{equation*}
where $\ell>0$ is a preset length chosen by the practitioner. 

In the Gaussian framework, based on the plug-in principle, we then build a general procedure that estimate the optimum and prove that the resulting empirical predictor performs as well as $\Gamma^*_{\ell}$ both in term of expected length and error rate.
Notably, the control on the expected length of the proposed estimator is distribution-free. Furthermore, our algorithm has two appealing properties. It can benefit from a semi-supervised setting and can be applied to any off-the-shelf machine learning algorithm. 


On the other hand, we evaluate the performance of our estimator with respect to the symmetric difference distance and a risk measure which properly balances the expected length and the error rate. Specifically, we establish  the consistency for our procedure under mild assumptions and provide rates of convergence under suitable assumptions on the distribution of the data.

We additionally conduct a numerical study that confirms our theoretical findings and shows how effective our method is in controlling the length, an important aspect to ensure the interpretability of the output. Finally, we provide a numerical  comparison with the strategy which consists in building PI under expected coverage constraint. Our numerical experiment highlights that our proposed approach produced significantly more stable PI. In particular, our algorithm seems to be  more adapted when the sample size of the training sample is moderate.

\paragraph{Related works. }
A first line of work related to PI is \emph{confidence intervals}. This is one of the most popular tools in statistical inference and differs from PI by the fact that the purpose there is to output a range of values for a given parameter of the model such that the mean, while our goal is here the prediction. The spectrum of applications of confidence interval is extremely wide and from some perspective PI can be seen as part of the confidence interval literature where we focus on building a confidence interval for the output of a new observation.


Probably the closest direction of works to ours is \emph{conformal prediction}~\cite{Vovk_Nouretdinov_Gammerman09, Lei_Wasserman13, Lei_G'Sell_Rinaldo_Tibshirani_Wasserman18}. The main difference relies on the way the expected length and the error rate of the prediction interval is considered. The goal there is to produce a PI with a pre-specified level of accuracy. The connection of PI with controlled expected size is important to figure out since, \emph{at the population level}, each PI with controlled accuracy corresponds to a PI with controlled expected size. From practice however, the two approaches start to differ. We defer this discussion to Sections~\ref{subsec:conformal} and~\ref{subsec:numCompaLei} where a complete comparison to \emph{PI with controlled accuracy} is conducted.

Providing an output with a pre-defined length has rarely been considered. Probably the first reference that deals with such notion is~\cite{Koopmans71}. There, the authors build confidence intervals for the mean and variance in a Gaussian problem that reach given confidence level while being of size $L$. In contrast to that work, we deal with prediction intervals, our control on the length is in expectation which offers more flexibility on ``hard'' points, we do not focus on a pre-specified level of confidence but rather minimize the error under a size constraint, and we derive a statistical and a numerical analysis of our method.

Finally, let us notify that constraining the expected length is not novel. It has already been considered in 
the multi-class classification setting~\cite{Denis_Hebiri17,chzhen_SetvaluedMinimax_21}. There, the control of the length is interpreted as the desired average number of output labels. Similar to the present work, the goal is to focus on a set of values for prediction while maintaining the interpretability  of the output. The main difference with earlier work is that we deal here with real valued output which is more tricky. From this perspective
the present paper is a generalization of these previous works to the Gaussian regression setting.

\paragraph{Outline of the paper. } Section~\ref{sec:framework} provides the main notation and describes the framework of prediction intervals under expected length constraint in the Gaussian regression.
In particular, the explicit form of the optimal rule is provided.
Section~\ref{sec:estimation} introduces our data-driven procedure as well as its statistical analysis. This theoretical analysis is complemented with a numerical study
presented in Section~\ref{sec:numResult}. Additional considerations beyond the Gaussian assumption and other frameworks of prediction intervals are considered in Section~\ref{sec:disc}. A conclusion is provided in Section~\ref{sec:conclusion}, while the proofs of our results are postponed to the Appendix.

\section{General framework}
\label{sec:framework}

In the present contribution we focus on the Gaussian model, that is, we assume that $(\bX,Y) \in\bbR^d \times \bbR$ are such that
\begin{equation}
\label{eq:eqModel}
    Y=f^{*}(\bX)+\sigma(\bX) \, \varepsilon\enspace ,
\end{equation}
where $\varepsilon \sim \mathcal{N}(0,1)$ is independent of $\bX$. In this expression, $f^*:\bbR^d \to  \bbR$ is the regression function and $\sigma:\bbR^d \to  \bbR_{+}^{*}$ is the conditional variance function, both of them assumed to be unknown. The main assumptions that we consider throughout the paper are presented in Section~\ref{subsec:ass}. The characterization of the optimal prediction interval under expected length constraint is provided in Section~\ref{subsec:optimalRule}. Finally, we define the measure of performance dedicated to asses the quality of a prediction interval in Section~\ref{subsec:perfMeasure}.

\subsection{Assumptions}
\label{subsec:ass}

Given an observation $\bX \in \bbR^d $, our goal is to produce the most accurate, in a certain sense to be specified later, a range of predicted values where the corresponding label $Y\in \bbR$ lies. 
Such predictions will be describe a set of $\mathcal{P}\left(\mathbb{R}\right)$ and denoted by $\Gamma(\bx)$ for each $\bx \in \bbR^d$. In other words, the predictor $\Gamma$ is a mapping from $\mathbb{R}^d$ onto $\mathcal{P}\left(\mathbb{R}\right)$. 

Throughout the paper we denote by $p(\cdot|\bx)$ the conditional density of $Y$ given $\bx$, that is, for all $y\in\bbR$
\begin{equation*}
p(y|\bx) = \dfrac{1}{\sqrt{2 \pi}\sigma(\bx)}    \exp\left( - \dfrac{\left(y-f^*(\bx)\right)^2}{2 \sigma^2(\bx)}
\right)\enspace , 
\end{equation*}
that is, we focus on the heteroscedastic Gaussian regression model. In order to avoid pathological situations, we impose the following mild assumptions on the regression and conditional variance functions.
\begin{assumption}
\label{ass:assSigma}
There exist $0 < \sigma_0 < \sigma_1 < \infty$ such that for all $\bx \in \mathbb{R}^d$
\begin{equation*}
 \sigma_0 \leq \sigma(\bx) \leq \sigma_1\enspace.   
\end{equation*}
\end{assumption}
\begin{assumption}
\label{ass:assFstar}
There exists $C_1 > 0$ such that
\begin{equation*}
\mathbb{E}\left[|f^*(\bX)|\right] \leq C_1 \enspace.    
\end{equation*}
\end{assumption}
In addition, we consider an assumption which is PI context-specific. It ensures in particular the existence and uniqueness and the optimal PI. Note that similar assumption is considered in the set-valued classification framework~\cite{chzhen_SetvaluedMinimax_21}.
\begin{assumption}[Continuity]
\label{ass:CDcontinuity}
For all $y\in \mathbb{R}$, the mapping $ t\mapsto \mathbb{P}_{\bX}(p(y|\bX)\geq t)$ is continuous on $\mathbb{R}^{*}_{+}$.
\end{assumption}
In other word, we assume that $p(y|\bX)$ is atomless.

\subsection{Prediction interval with expected length}
\label{subsec:optimalRule}

For a given predictor $\Gamma$ two features are of interest, its error rate $\mathbb{P}(Y \notin \Gamma(\bX))$
and its expected Lebesgue measure defined as
\begin{equation*}
\mathcal{L}(\Gamma) := \mathbb{E}\left[L(\Gamma(\bX)\right] = \mathbb{E}\left[\int_{\bbR} \one_{\{y\in \Gamma(\bX)\}}   \mathrm{d}y\right]\enspace.
\end{equation*}
Given $\ell > 0$, we focus on the following problem
\begin{equation}
\label{eq:eqOracle}
\Gamma^{*}_{\ell} \in \arg\min\{\mathbb{P}\left(Y \notin \Gamma(\bX)\right) \; : \; \Gamma \; \text{such that } \mathcal{L}(\Gamma) \leq \ell\}\enspace.
\end{equation}
The next proposition provides the characterization of the optimal predictor under Assumption~\ref{ass:CDcontinuity}.

\begin{proposition}
\label{prop:oraclePredictor}
Let $\ell > 0$, under Assumption~\ref{ass:CDcontinuity}, the optimal predictor $\Gamma^*_{\ell}$ can be expressed as
\begin{eqnarray*}
\Gamma^{*}_{\ell}(\bX) = \{y\in \bbR : \ p(y|\bX)\geq \lambda^*_\ell\} \enspace,
\end{eqnarray*}
where $\lambda^*_\ell=G^{-1}(\ell)$ with $G(t):=\int_{\mathbb{R}} \mathbb{P}(p(y|\bX)\geq t)dy$ for all $t > 0$\footnote{When $t=0$, we have $G(t)=+\infty$ and then we will use the convention $G^{-1}(+\infty) = 0$.}.
\end{proposition}
The parameter $\lambda^*_\ell$, which corresponds to the value of the generalized inverse function $G^{-1}$ at $\ell$, plays a crucial role in our study since it fully determines the optimal predictor $\Gamma^*$. 
This being said, let us comment on Proposition~\ref{prop:oraclePredictor}.
First, an important consequence of the above proposition is that the predictor $\Gamma_{\ell}^*$ is an interval of length $\ell$, that is $\mathcal{L}( \Gamma^*_\ell) = \ell$ and we additionally can express $\Gamma^*_\ell$ as
\begin{equation*}
\Gamma_{\ell}^*(\bX) =  \left[ f^*(\bX)-\sqrt{2\sigma^{2}(\bX)\log\left(\frac{1}{\sqrt{2\pi}\lambda^{*}_{\ell}\sigma(\bX)}\right)} \ , \   f^*(\bX)+\sqrt{2\sigma^{2}(\bX)\log\left(\frac{1}{\sqrt{2\pi}\lambda^{*}_{\ell}\sigma(\bX)}\right)} \   \right]\enspace.   
\end{equation*}
Second, the function $G$ defined in Proposition~\ref{prop:oraclePredictor} is the extension to the regression case of the function $G$ defined in~\cite{Denis_Hebiri17} in the multi-class setting.  Note that the function $G$ is always well-defined and continuous for $t>0$, since by Markov Inequality and Fubini Theorem, 
\begin{equation*}
G(t) = \int_{\mathbb{R}} \mathbb{P}(p(y|\bX)\geq t)dy \leq \frac{1}{t} \int_{\mathbb{R}} \mathbb{E}\left[p(y|\bX)\right]\mathrm{d}y \leq \frac{1}{t} \mathbb{E}\left[\int_{\mathbb{R}} p(y|\bX) \mathrm{d}y\right] \leq \frac{1}{t} \enspace.
\end{equation*}
Finally, we highlight that parameter $\lambda_{\ell}^*$ is simply the Lagrange multiplier of the minimization problem defined by Equation~\eqref{eq:eqOracle}. Therefore, $\Gamma_{\ell}^*$ can be expressed as the minimizer of the unconstrained problem
\begin{equation}
\label{eq:eqUnconstrained}
\Gamma^*_{\ell} \in \arg\min_{\Gamma} \mathbb{P}\left(Y \notin \Gamma(\bX)\right) + \lambda_{\ell}^* \mathbb{E}\left[L(\Gamma(\bX)\right]\enspace.
\end{equation}

\subsection{Measures of performance}
\label{subsec:perfMeasure}

In this paragraph we introduce two ways to quantify the quality of a given prediction interval $\Gamma$ . 
The first one, suggested by Equation~\eqref{eq:eqUnconstrained}, balances the error rate and the expected length of the predictor
\begin{equation*}
R_{\ell}(\Gamma) =  \mathbb{P}\left(Y \notin \Gamma(\bX)\right) + \lambda_{\ell}^* \mathbb{E}\left[L(\Gamma(\bX)\right]\enspace,
\end{equation*}
with $\lambda_{\ell}^* = G^{-1}(\ell)$.
This risk is particularly important from our perspective since minimizing it over all predictors lead to the optimal predictor $\Gamma_{\ell}^*$, which reaches the requested expected length.
A natural ``distance'' to the optimal predictor is then evaluated through the excess risk
\begin{equation*}
\mathcal{E}_{\ell}\left(\Gamma\right) = R_{\ell}(\Gamma)-R_{\ell}(\Gamma^*_{\ell}) \enspace.    
\end{equation*}
The following proposition provides a closed formula for this term.
\begin{proposition}
\label{prop:propExcessRisk}
 Let $\ell\geq 0$. For any predictor $\Gamma$
\begin{equation*}
\mathcal{E}_{\ell}(\Gamma)=\mathbb{E}\left[\int_{\Gamma(\bX)\triangle \Gamma_{\ell}^{*}(\bX)} \left|p(y|\bX)-\lambda^*_\ell \right| \ dy \right]\enspace .
\end{equation*}
\end{proposition}
Interestingly, the above result shows that the performance of a predictor $\Gamma$ is directly linked to the behavior of the conditional density $p(y|\bx)$ around the threshold $\lambda_{\ell}^*$ on the symmetric difference 
$\{\Gamma(\bX)\triangle \Gamma_{\ell}^{*}(\bX)\}$.

A second measure of performance arises naturally when we deal with predictors that are intervals. It is the expectation of symmetric difference between the considered predictor $\Gamma$ and optimal predictor $\Gamma_{\ell}^*$ defined for all predictor $\Gamma$ as
\begin{equation*}
\mathcal{H}\left(\Gamma\right) =  \mathbb{E}\left[L\left(\Gamma(\bX)\triangle \Gamma_{\ell}^{*}(\bX)\right)\right]
=  \mathbb{E}\left[\int_{\Gamma(\bX)\triangle \Gamma_{\ell}^{*}(\bX)} dy \right] \enspace.
\end{equation*}
In some sense, we note that the measure $\mathcal{H}$ provides a stronger guarantee than the excess risk since
$\mathcal{E}_{\ell}(\Gamma) \leq C_2 \mathcal{H}(\Gamma)$ where $C_2$ is a positive constant which depends on $\sigma_0$. 
Besides, $\mathcal{H}(\Gamma) = 0$ implies that $\Gamma = \Gamma_{\ell}^*$ while this property does not necessarily hold for the excess risk.

\section{Data-driven procedure}
\label{sec:estimation}

In this section, we provide a  general data-driven procedure to estimate the optimal predictor $\Gamma_{\ell}^*$.
Two key features are expected from the resulting empirical prediction interval. The expected length should be of order $\ell$
while keeping its error rate close to one obtained by the oracle predictor.
The estimation procedure is presented in the Section~\ref{subsec:dataDrivenProce},  and its main properties are
provided in Section~\ref{subsec:Consistency}. Finally, Section~\ref{subsec:ratesOfConvergence} is dedicated to the study of rates of convergence.

\subsection{Empirical prediction interval}
\label{subsec:dataDrivenProce}

The result provided in Proposition~\ref{prop:oraclePredictor} suggests that an empirical prediction interval can be obtained through the plug-in principle by considering estimators of the conditional density $p$ and the parameter $\lambda^*_{\ell} = G^{-1}(\ell)$. From a theoretical perspective, this learning task requires two independent samples.

First, in order to build an estimator of the conditional density $p$, we estimate the functions $f^*$ and $\sigma$. Hence, we exploit a labeled sample $\mathcal{D}_{n}=\{(\bX_i,Y_i)\}_{i=1}^{n}$ and build based on it estimators $\hat{f}$ and $\tilde{\sigma}$ of these two functions by the means of any machine learning algorithm.
However, to establish theoretical guarantees, we require that the estimator $\tilde{\sigma}$ satisfies similar assumption as Assumption~\ref{ass:assSigma}. To this end, we consider a thresholded version of the estimator $\tilde{\sigma}$ denoted by $\hat{\sigma}$ and define for $s > 0$ as
\begin{equation*}
\hat{\sigma}^2(\bx) = \tilde{\sigma}^2(\bx) \one_{\{s^{-1}\leq\tilde{\sigma}^2(\bx) \leq s\}} + s^{-1}  \one_{\{\tilde{\sigma}^2(\bx) < s^{-1}\}} 
+ s\one_{\{\tilde{\sigma}^2(\bx) > s\}} \enspace .
\end{equation*}
A straightforward consequence of the definition of $\hat{\sigma}$ is that $\frac{1}{s} \leq \hat{\sigma}^2(\bx)\leq s$.    
Furthermore, if $s$ satisfies $\frac{1}{s} \leq \sigma_0^2 \leq \sigma_1^2 \leq s$, we have for all $\bx$
\begin{equation*}
\left|\hat{\sigma}^2(\bx)-\sigma^2(\bx) \right| \leq   \left|\tilde{\sigma}^2(\bx)-\sigma^2(\bx) \right| \enspace ,
\end{equation*}
Hence consistency of $\tilde{\sigma}^2$ would imply the consistency of $\hat{\sigma}^2$.

Based on $\hat{f}$ and $\hat{\sigma}$, an estimator $\tilde{p}$ of the conditional density $p$ naturally derives and can be written for all $(\bx,y)\in \bbR^d\times \bbR$ as
\begin{equation*}
\tilde{p}(y|\bx) = \dfrac{1}{\sqrt{2 \pi}\hat{\sigma}(\bx)}    \exp\left( - \dfrac{\left(y-\hat{f}(\bx)\right)^2}{2 \hat{\sigma}^2(\bx)}\right)\enspace.
\end{equation*}
The second step is devoted to the estimation of the parameter $\lambda^*_{\ell}$ and requires an \emph{unlabeled} sample $\mathcal{D}_N = \{\bX_{n+1}, \ldots, \bX_{n+N}\}$ which consists of \iid observations of $\bX$ and is independent of $\mathcal{D}_n$. 
Since $\lambda^*_{\ell}$ depends on the function $G$, it is suitable to consider the empirical counterpart of the function $G$, that we build based on $\hat{p}$ and define for all $t\in [0,1]$ as
\begin{equation*}
\tilde{G}(t) = \int_{\mathbb{R}} \frac{1}{N}\sum_{i = 1}^{N} \one_{\{\hat{p}(y|\bX_{n+i}) > t\}} \rm{d} y \enspace.
\end{equation*}
As a result, the empirical prediction interval is defined\footnote{Here again, we use the convention $\tilde{G}^{-1}(+\infty) = 0$.} point-wise as
\begin{equation*}
\tilde{\Gamma}(\bx) = \{y \in \mathbb{R}: \ \tilde{p}(y|\bx) \geq \tilde{G}^{-1}(\ell)\}\enspace.    
\end{equation*}
The predictor $\tilde{\Gamma}$ is very natural but has a few limitations: i) because $Y$ is unbounded, the study of the theoretical properties of the estimator $\tilde{\Gamma}$ might be difficult; ii) in addition, establishing a theoretical analysis on $\tilde{\Gamma}$ involves similar assumption to Assumption~\ref{ass:CDcontinuity} for $\tilde{G}$. More precisely, it requires that conditional on $\mathcal{D}_n$ the cumulative distribution of $\tilde{p}(y|\bX)$ is atomless;  iii) furthermore, the above expression of $\tilde{\Gamma}(\bx)$ is explicit but relies on computing an integral in order to evaluate the function $\tilde{G}$. This integral should be approximated.
To circumvent all these issues, we consider the following modifications of the initial estimator $\tilde{\Gamma}$.

\paragraph*{For i) -- Thresholding.}

Let $s > 0$,
we  consider a thresholded version of $p$ given by
\begin{equation}
\label{eq:EstimateurSueille}
\hat{p}(y|\bx) = \dfrac{1}{\sqrt{2 \pi}\hat{\sigma}(\bx)}    \exp\left( - \dfrac{\left(y-\hat{f}(\bx)\right)^2}{2 \hat{\sigma}^2(\bx)}\right) \one_{\{|y| \leq s\}}\enspace .
\end{equation}

\paragraph*{For ii) -- Randomization.}

To ensure the continuity of the conditional C.D.F. of $\hat{p}(y|\bX)$ for $y \in [-s,s]$,
we introduce a random perturbation $\zeta$ distributed according to a Uniform
distribution on $[0,u]$, for $u > 0$ and independent of $(\bX,Y)$.
We then define the randomized version of $\hat{p}$ as
\begin{equation}
\label{eq:addpertubation}
\hat{p}(y|\bX, \zeta) = \hat{p}(y|\bX) + \zeta   \one_{\{|y| \leq s\}} \enspace.
\end{equation}

\paragraph*{For iii) -- Discretization.}

To approximate $\tilde{G}$, we simply consider the Riemann sum  based on the regular grid $\mathcal{G}  = \{y_1, \ldots, y_M\}$ of $[-s,s]$ for some $M \geq 1$. To this end, we introduce $(\zeta_1, \ldots, \zeta_N)$ i.i.d. copies of $\zeta$ and then define
\begin{equation*}
\hat{G}(t) = \frac{2s}{MN}\sum_{k = 1}^M \sum_{i = 1}^N \one_{\{\hat{p}(y_k|\bX_{n+i}, \zeta_i) > t\}} \enspace.    
\end{equation*}
Finally, the resulting empirical prediction interval writes as
\begin{equation}
\label{eq:Final_Predictor_Randomized}
\hat{\Gamma}(\bX,\zeta) = \{y \in \mathbb{R}: \ \hat{p}(y|\bX,\zeta) \geq \hat{G}^{-1}(\ell)\}   \enspace . 
\end{equation}


\subsection{Theoretical guarantees}
\label{subsec:Consistency}

In this section, we provide the main properties of the empirical prediction interval $\hat{\Gamma}$.
We first illustrate that the prediction interval $\hat{\Gamma}$ has an expected length equal to the requested value $\ell$. This is one of the main striking feature of our data-driven procedure.
\begin{proposition}
\label{prop:propLength}
Assume that $M > 4\sqrt{N}$, then
\begin{equation*}
\mathbb{E}\left[\left|\mathcal{L}(\hat{\Gamma})-\ell\right|\right]  
\leq C\dfrac{s}{\sqrt{N}}\enspace,
\end{equation*}
where $C>0$ is an absolute constant.
\end{proposition}
The above result states that our methodology is able to produce a prediction interval with an expected length $\ell$, irrespectively of the distribution of the data and of whether or not we have build accurate estimates for $f^*$ and $\sigma$. Importantly, Proposition~\ref{prop:propLength} holds even if $(\bX,Y)$ does not satisfy Equation~\eqref{eq:eqModel}.
From this perspective the control on the expected length of the produced prediction interval is \emph{distribution-free}.
Notice in particular that the stated bound depends only on the parameter $s$ which should be specified by the practitioner (this choice is discussed later) and on the number $N$ of unlabeled data. In some semi-supervised applications, the amount of these data can be very large so that we can get a good approximation of the marginal distribution $\mathbb{P}_{\mathbf{X}}$ and then we can expect a good control of the expected length almost for free.
Let us also add that Proposition~\ref{prop:propLength} is a fundamental step to show the following bound on the excess risk:
\begin{proposition}
\label{prop:propConsist}
Let Assumption~\ref{ass:CDcontinuity} be satisfied. For $M > 4\sqrt{N}$, we have
\begin{equation*}
\mathbb{E}\left[\mathcal{E}_\ell\left(\hat{\Gamma} \right)\right] \leq C \left( \mathbb{E}\left[\int_{\mathbb{R}}\left|\hat{p}(y|\bX)-p(y|\bX)\right|{\rm d}y\right] + su + \dfrac{s}{\sqrt{N}} \right)\enspace,
\end{equation*}
where $C>0$ is an absolute constant.
\end{proposition}
The above result shows that the excess-risk of $\hat\Gamma$ is mainly controlled by the $L_1$-risk of the estimator of the conditional density. The residual terms are related to the randomization on the one hand and to the control of the expected length of $\hat\Gamma$, given in Proposition~\ref{prop:propLength}, on the other hand. 
Proposition~\ref{prop:propConsist} is an intermediate step to establish consistency of the proposed prediction interval as well as to build explicit rates of convergence for the excess-risk of $\hat\Gamma$. This is the purpose of the next paragraph and Section~\ref{subsec:ratesOfConvergence} respectively.

\paragraph*{Consistency result.}

Proposition~\ref{prop:propConsist} shows that the consistency of $\hat{\Gamma}$ with respect to the excess-risk relies to the consistency of the estimator $\hat{p}(y|\bx)$. In view of Equation~\eqref{eq:EstimateurSueille}, it is clear that the performance of $\hat{p}$ is directly linked to the statistical properties of $\hat{f}$ and $\hat{\sigma}$. More precisely, we obtain the following result.

\begin{theorem}
\label{thm:consistance}
Let Assumptions~\ref{ass:assSigma},~\ref{ass:assFstar}, and~\ref{ass:CDcontinuity}. Consider $s = \log(\min(n,N))$, $M > 4\sqrt{N}$, and $u = u_n \rightarrow 0$. Assume that
\begin{equation*}
\sqrt{s}\mathbb{E}\left[(\hat{f}(X)-f^*(X))^2\right] \rightarrow 0, \;\; {\rm and} \;\;  s^{5/2}\mathbb{E}\left[|\hat{\sigma}^2(X)-\sigma(X)|\right] 
\rightarrow 0 \enspace ,
\end{equation*}
then the following holds
\begin{equation*}
\mathbb{E}\left[\mathcal{E}_{\ell}\left(\hat{\Gamma}\right)\right] \leq C_2\mathbb{E}\left[\mathcal{H}\left(\hat{\Gamma}\right)\right] \rightarrow 0\enspace .
\end{equation*}
\end{theorem}
Let us make several comments on this theorem.
First, under suitable assumptions, both excess-risk and expected symmetric difference of $\hat{\Gamma}$ converge to $0$. Notably, since $\mathbb{E}\left[\mathcal{E}_{\ell}\left(\hat{\Gamma}\right)\right]  \leq C_2\mathbb{E}\left[\mathcal{H}\left(\hat{\Gamma}\right)\right]$, consistency \emph{w.r.t.} the expected symmetric difference implies consistency \emph{w.r.t.} the excess-risk. From this perspective, symmetric difference control is a more difficult problem that excess-risk control. In particular, $\mathbb{E}\left[\mathcal{H}\left(\hat{\Gamma}\right)\right] \rightarrow 0$ indicates that $\hat{\Gamma} = \Gamma_{\ell}^*$ asymptotically.
Another aspect that need to be discussed is the assumptions that are requested for the proof of Theorem~\ref{thm:consistance}. 
More specifically, consistency of $\hat{f}$, and $\hat{\sigma}^2$ are naturally required to ensure that $\hat{p}$ is a consistent estimator of $p$. In particular, convergence of $\hat{f}$ and $\hat{\sigma}$ can be made possible by several learning algorithms such as kernel methods, local polynomials, regularized least-squares among many others.


\subsection{Rates of convergence}
\label{subsec:ratesOfConvergence}

Theorem~\ref{thm:consistance} establishes the consistency of the prediction interval $\hat{\Gamma}$ under mild assumptions. 
In this section, we focus on rates of convergence. More structural assumptions are then required. We borrow conditions from~\cite{Denis_Hebiri_Zaoui20} introduced in the framework of regression with abstention. We assume that $\bX$ belongs to a compact $\mathcal{C}$, and we consider the following assumptions.
\begin{assumption}[Regularity]
\label{ass:regularity}
The functions $f^*$ and $\sigma^2$ are Lipschitz.
\end{assumption}
\begin{assumption}[Strong density assumption]
\label{ass:StrongDensityAssumption}
The marginal distribution $\mathbb{P}_{\bX}$ satisfies the strong density assumption
\begin{itemize}
\item[$\bullet$] $\mathbb{P}_{\bX}$ is supported on a compact regular set $\mathcal{C}\subset \mathbb{R}^d $,
\item[$\bullet$]$\mathbb{P}_\bX$ admits a density $\mu$ w.r.t. to the Lebesgue measure such that  $0 < \mu_{\min} \leq \mu(\bx) \leq \mu_{\max}<\infty$, for all $\bx \in \mathcal{C}$.
\end{itemize}
\end{assumption}
\begin{assumption}[$\alpha$-Margin assumption]
We say that $p(\cdot|X)$ satisfies Margin assumption with parameter $\alpha\geq 0$ at level $\lambda_{\ell}$ with respect to $\mathbb{P}_X$ if there exist constants $c_0>0$ and $t_0>0$ such that for all $0<t\leq t_0$, 
\begin{equation*}
\int_{\mathbb{R}}\mathbb{P}_{X}\left( |p(y|\bX)-\lambda_{\ell}|\leq t\right) dy \leq c_0 t^{\alpha}\enspace.
\end{equation*}
\label{ass:marginAss}
\end{assumption}
The above first two assumptions are rather classical when we deal with rates of convergence in nonparametric statistic. We refer the reader to the book~\cite{Gyofri_Kohler_Krzyzak_Walk02} for a more detail discussion. In addition, Assumption~\ref{ass:marginAss}, also known as Tsybakov noise condition~\cite{Tsybakov04}, has been introduced in the binary classification setting to get fast rates of convergence~\cite{Audibert_Tsybakov07}.
In our setting, we notice that the Tsybakov noise condition is required around the threshold $\lambda_{\ell}$. Moreover, since we extend this assumption to the case of regression, we need to ingrate it \wrt $y\in\bbR$.
Based on the above conditions, we can establish the following result.
\begin{proposition}
\label{prop:excessRiskUnderMargin}
Let Assumptions~\ref{ass:assSigma},~\ref{ass:regularity},~\ref{ass:StrongDensityAssumption}, and~\ref{ass:marginAss} be satisfied. 
For $s = \log(\min(n,N))$, and $M > 4 \sqrt{N}$, we have that
\begin{equation*}
\mathbb{E}\left[\mathcal{E}_{\ell}\left(\hat{\Gamma}\right)\right] \leq 
C\left(\mathbb{E}\left[\left(\sup_{(x,y) \in \mathcal{C} \times [-s,s]} \left|\hat{p}(y|x) -p(y|x) \right|\right)^{1+\alpha} \right] + \frac{1}{\min(n,N)^{1+\alpha}} + u^{1+\alpha} + \dfrac{\log(N)}{\sqrt{N}}\right),
\end{equation*}
where $C > 0$ is a constant which depends on $f^*$, $\sigma^2$, $c_0$, $\alpha$, and $\mathcal{C}$. 
\end{proposition}
As compared to the upper-bound that we get in Proposition~\ref{prop:propConsist}, the bound here is better because of the exponent $1+\alpha$ against $1$. However, it is obtained under stronger assumptions.

\paragraph*{Estimators of regression and variance function.}
The framework that we have described so far is quite general and allows to use any off-the-shelf machine learning algorithms to estimate the regression and the variance functions. In what follows, we propose a more concrete illustration of our approach by considering empirical prediction intervals $\hat{\Gamma}$
where both regression and variance functions are estimated with the $k$NN algorithm. 
Hereafter, we briefly recall the definition of the estimators that are based on the {\emph labeled} sample $\mathcal{D}_n$.
For any $\bx \in \mathbb{R}^d$, we denote by $(\bX_{(i,n)}(\bx),Y_{(i,n)}(\bx)), i = 1,\ldots n$ the reordered data according to the $\ell_2$ distance in $\mathbb{R}^d$, meaning that 
$$\|  \bX_{(i,n)}(\bx) - \bx  \| < \|  \bX_{(j,n)}(\bx) - \bx  \| \enspace ,$$
for all $i< j$ in $\{1,\ldots, n \}$.
For simplicity, we assume that ties occur with probability $0$.
Let $k = k_n$ be an integer.
The $k$NN estimator of $f^*$ and $\sigma^2$ are then defined, for all $\bx\in \bbR^d$, as follows
\begin{equation}
\label{eq:knnestimators}
\hat{f}(\bx) = \frac{1}{k_n}\sum_{i = 1}^{k_n}  Y_{(i,n)}(\bx) \;\;{\rm and} \;\; \tilde{\sigma}^2(\bx) = \frac{1}{k_n} \sum_{i =1}^{k_n}\left( Y_{(i,n)}(\bx) - \hat{f}(X_{(i,n)}(\bx))\right)^{2} \enspace.
\end{equation}
The properties of these estimator are provided in~\cite{Gyofri_Kohler_Krzyzak_Walk02} for the regression function and in~\cite{Denis_Hebiri_Zaoui20} for the variance function.
In particular, the authors in~\cite{Denis_Hebiri_Zaoui20} establish rates of convergence {\it w.r.t.} the sup-norm for the estimator $\hat{\sigma}$.

\paragraph*{Rates of convergence.}
The next result, which is an adaption of Proposition~3.1 in~\cite{chzhen_SetvaluedMinimax_21}, is useful to derive upper-bound on the measure of risk $\mathcal{H}$ of $\hat{\Gamma}$ thanks to a control on the excess-risk.
\begin{proposition}
Let Assumptions~\ref{ass:marginAss} be satisfied. There exists an absolute constant $C_3>0$ such that
\label{prop:comparisonHammingRisk}
\begin{equation*}
\mathbb{E}\left[\mathcal{H}(\hat{\Gamma})\right] \leq C_3 \left(\mathbb{E}\left[\mathcal{E}_{\ell}\left(\hat{\Gamma}\right)\right]\right)^{\alpha/\alpha+1} \enspace.
\end{equation*}
\end{proposition}
Importantly, this proposition, together with the inequality $\mathcal{E}_{\ell}\left(\Gamma\right) \leq C_2\mathcal{H}({\Gamma}) $ for all $\Gamma$, shows that under appropriate regularity condition consistency of $\hat{\Gamma}$ \emph{w.r.t.} the distance $\mathcal{H}$ and the excess-risk are equivalent. The only difference is in the rates of convergence. The above result highlights the link between them under Assumption~\ref{ass:marginAss}.
In particular, we only have to establish rates of convergence {\it w.r.t.}~$\mathcal{E}$.
Let us introduce the following notation. When $ a \propto b$, it means that the quantities $a$ and $b$ are equal up to a constant. Moreover  $\lesssim_{\log(n)}$ says that the inequality holds up to some constants and logarithmic factors. Now, we state the main result of this section. 
\begin{theorem}
\label{thm:RatesOfcve}
Let Assumptions~\ref{ass:assSigma} and~\ref{ass:regularity}-\ref{ass:marginAss} be satisfied.
Let $k_n \propto n^{-2/d+2}$,
$s = \log(\min(n,N))$, $M > 4\sqrt{N}$, and $u_n = \frac{1}{n}$. 
The following holds
\begin{equation*}
\mathbb{E}\left[\mathcal{E}_{\ell}(\hat{\Gamma})\right]
\lesssim_{\log(n)} n^{-(1+\alpha)/(d+2)} + \min(n,N)^{-(1+\alpha)}+ N^{-1/2} \enspace .
\end{equation*}
\end{theorem}
Several comments can be made from the above result.
The first term is the classical nonparametric fast rate of convergence for the excess-risk under the Margin assumption and the Lipschitzness of the regression function. The last two terms that are related to the problem of PI estimation have different behavior according to the interplay between $n$ and $N$. In particular, as soon as $N\leq n$, the limiting term is $N^{-1/2}$ and the rate becomes slow if $n^{-(1+\alpha)/(d+2)} $ goes faster to $0$. On the other hand, if the number of unlabeled data $N$ is large with $N \gg  n^{1+\alpha}$ we recover the fast rate of convergence $n^{-(1+\alpha)/(d+2)} $. Between these two extremes, $N^{-1/2}$ can still be the limiting term. However, we hope that in our semi-supervised setting, enough data are available to make this term negligible as compared to the others.

\section{Extension and other approach}
\label{sec:disc}
In this section, we discuss some points beyond the considered framework in this paper. The extension of our results to other regression models is presented in Section~\ref{subsec:extension}. Another approach to build prediction interval based on the control of the expected error rate~\cite{Lei_Wasserman13} is described in Section~\ref{subsec:conformal}. In particular, we exhibit the main differences with our considered procedure.

\subsection{Beyond Gaussian setting}
\label{subsec:extension}

In the present work, we study prediction intervals under expected length constraint in the heteroscedastic Gaussian regression setup. The appealing aspect of this framework lies in the form of the optimal predictor 
\begin{eqnarray}
\label{eq:optimalrulediscussion}
\Gamma^{*}_{\ell}(\bX) = \{y\in \bbR : \ p(y|\bX)\geq \lambda^*_\ell\} \enspace,
\end{eqnarray}
with $\lambda^*_\ell=G^{-1}(\ell)$ and $G(t):=\int_{\mathbb{R}} \mathbb{P}(p(y|\bX)\geq t)dy$. Furthermore, the density $p(y|\bX)$ has an explicit expression that exclusively depends on the regression and the conditional variance functions $f$ and $\sigma$. Therefore, our proposed algorithm only involves estimators of $f$ and $\sigma$ to estimate the conditional density $p$. In particular, we do not consider any general procedure for density estimation. 

In this paragraph, we discuss possible extensions outside the Gaussian framework but still considering the regression framework 
$Y=f^{*}(\bX)+\sigma(\bX) \, \varepsilon$.
In order to make sure that the optimal predictor is well defined, we require the following assumption.
\begin{assumption}
\label{ass:Conddensity}
We assume that the variable $Y$ given $\bX$ has density $p(\cdot|\bX)$.
\end{assumption}
If we do not assume that $Y|\bX $ belongs to a given family of distribution, the characterization of the prediction interval~\eqref{eq:optimalrulediscussion} still holds but the expression of the conditional density can not be simplify. Therefore, a data-driven predictor, based on the plug-in principle, must rely on estimates $\hat{p}(\cdot | \bx)$ of the conditional density $p(\cdot | \bx)$. The way to build the estimator $\hat{\Gamma}$ does not differ from the Gaussian case ones $\hat{p}$ is obtained (see Section~\ref{sec:estimation}). From the theoretical perspective, general properties such as Propositions~\ref{prop:oraclePredictor} and~\ref{prop:propExcessRisk} still hold and the question here is to investigate consistency results of the algorithm $\hat{\Gamma}$. The control on the expected length of the prediction interval $\mathbb{E}\left[\left|\mathcal{L}(\hat{\Gamma})-\ell\right|\right]  
\leq C\dfrac{s}{\sqrt{N}}$ given in Proposition~\ref{prop:propLength} is also still valid since this result is distribution-free. On the other hand, consistency for the excess-risk requires conditions. 
In the case where $Y$ is bounded, if the estimator of the conditional probabilities is such that 
$
\mathbb{E}\left[\int_{\mathbb{R}} |\hat{p}(y|\bX)-p(y|\bX)|dy\right] \underset{n\to +\infty}{\longrightarrow}0,
$ we can establish under Assumptions~\ref{ass:assSigma},~\ref{ass:assFstar},~\ref{ass:CDcontinuity}, and~\ref{ass:Conddensity}  that 
\begin{equation*}
\mathbb{E}\left[\mathcal{H}\left(\hat{\Gamma}\right)\right] \underset{n,N \to +\infty}{\longrightarrow}0 \enspace.    
\end{equation*}

Essentially, this result says that the estimation procedure that we study in this paper extends beyond the Gaussian setting. In particular, we still manage to get consistency for bounded random variable. It is worth mentioning that consistency might also be obtained as soon as $Y|\bX$ is sub-Gaussian. Then our method is statistically valid for general settings.

\subsection{Prediction interval under expected coverage constraint}
\label{subsec:conformal}

In this section, we present the approach which focuses on the construction of prediction interval under expected coverage. This method consists in minimizing the length of the prediction interval under a constraint on its expected error rate.
This approach is for instance studied in~\cite{Lei_Wasserman13}.

More precisely, let $\beta > 0$. We consider the following problem
\begin{equation*}
\Gamma^{*}_{\beta} = \arg\min_{\mathbb{P}(Y \notin \Gamma(\bX) \leq \beta}
\mathbb{E}\left[L(\Gamma(\bX)\right]]\enspace.
\end{equation*}
Under Assumptions~\ref{ass:CDcontinuity} and~\ref{ass:Conddensity} we can derive an expression of $\Gamma_{\beta}^*$ based on thresholding of the conditional densities:
\begin{equation*}
\Gamma^*_{\beta} = \{y \in \mathbb{R}, \;\; p(y|\bx) \geq t_{\beta}\} \enspace,    
\end{equation*}
with $t_{\beta}$ defined as solution of
\begin{equation*}
\mathbb{E}\left[\one_{\{p(Y|\bX) \geq t_{\beta} \}}\right] = \int_{\mathbb{R}} \one_{\{p(y|\bx) \geq t_{\beta}\}} p(y|\bx) \, {\rm d}y = 1-\beta \enspace. 
\end{equation*}
Therefore, from the above equation, we deduce that
\begin{equation*}
H^{-1}(t_{\beta}) =  1-\beta \enspace,    
\end{equation*}
where $H(t) = \mathbb{E}\left[\one_{\{p(Y|\bX) \geq t\}}\right]$.
Similarly to the procedure described in Section~\ref{subsec:dataDrivenProce}, we are able to provide a randomized prediction interval $\hat{\Gamma}_{\beta}$
based on the estimator $\hat{p}$. We point out that an important difference between the construction of estimators $\hat{\Gamma}_{\beta}$ and $\hat{\Gamma}$  is the estimation of the function $H$. Indeed, this step require a {\it labeled} and not an {\it unlabeled} dataset, but do not request the discretization step. More formally, considering a {\it labeled} dataset $\mathcal{D}_K = \{(\bX_i,Y_i), \ i = 1, \ldots, K\}$, and $(\zeta_1, \ldots, \zeta_K)$ the vector of perturbation, the estimator $\hat{H}$ of the function $H$ is defined for each $t > 0$, as follows 
\begin{equation*}
\hat{H}(t)  =   \frac{1}{K} \sum_{i = 1}^K \one_{\{p(Y_i|\bX_i,\zeta_i) \geq t\}}\enspace .
\end{equation*}
Although a theoretical comparison with our proposed method is not our purpose,
using similar arguments as in~\cite{Lei_Wasserman13}, we can establish the consistency of $\hat{\Gamma}_{\beta}$ under same assumptions as in Theorem~\ref{thm:consistance}. 
\begin{equation*}
\mathbb{E}\left[\mathcal{H}\left(\hat{\Gamma}_{\beta}\right)\right] \rightarrow 0 \enspace .
\end{equation*}
In Section~\ref{sec:numResult},  we focus on a comparison between our method and the expected coverage approach from a numerical perspective.

\section{Numerical experiments}
\label{sec:numResult}
This section is devoted to a numerical study of the performance of our procedure. More precisely, we analyze our approach on synthetic data in Section~\ref{subsec:simulation} and provide a comparison with the expected coverage approach described in Section~\ref{subsec:numCompaLei}.
 
\subsection{Simulation study}
\label{subsec:simulation}

We illustrate the performance of our procedure on the following model 
\begin{equation}
\label{eq:modelSimu}
Y = \exp(-\|\bX\|_2) +\frac{d\varepsilon}{2+4\|\bX\|_2}, \;\; \bX \in \mathbb{R}^d \enspace ,
\end{equation}
where $\bX =(X^1, \ldots, X^d)$ is such that for $j=1, \ldots, d$, the $X^j$ are \iid simulated according to a Uniform distribution on $[0,1]$ and are independent from $\varepsilon \sim \mathcal{N}(0,1)$. Note that the considered model satisfies Equation~\ref{eq:eqModel}, and that Assumptions~\ref{ass:assSigma}, and~\ref{ass:assFstar} are fulfilled.

For our numerical experiments, we choose reasonable dimensions of the features space $d \in \{1,5\}$.  
Before going further in our investigations, we display the boxplots of the output variable $Y$ in Figure~\ref{fig:boxplotModel}. We see that the range of values of $Y$ is much larger for $d=5$ and is included in $[-5,5]$ for both $d =1,5$. Besides, we chose to focus on $\ell \in \{0.1, 0.5, 1, 2\}$ which seems to be relevant values according to Figure~\ref{fig:boxplotModel} in order to still get interpretation of the output.
\begin{figure}
    \centering
    \includegraphics[scale= 0.5]{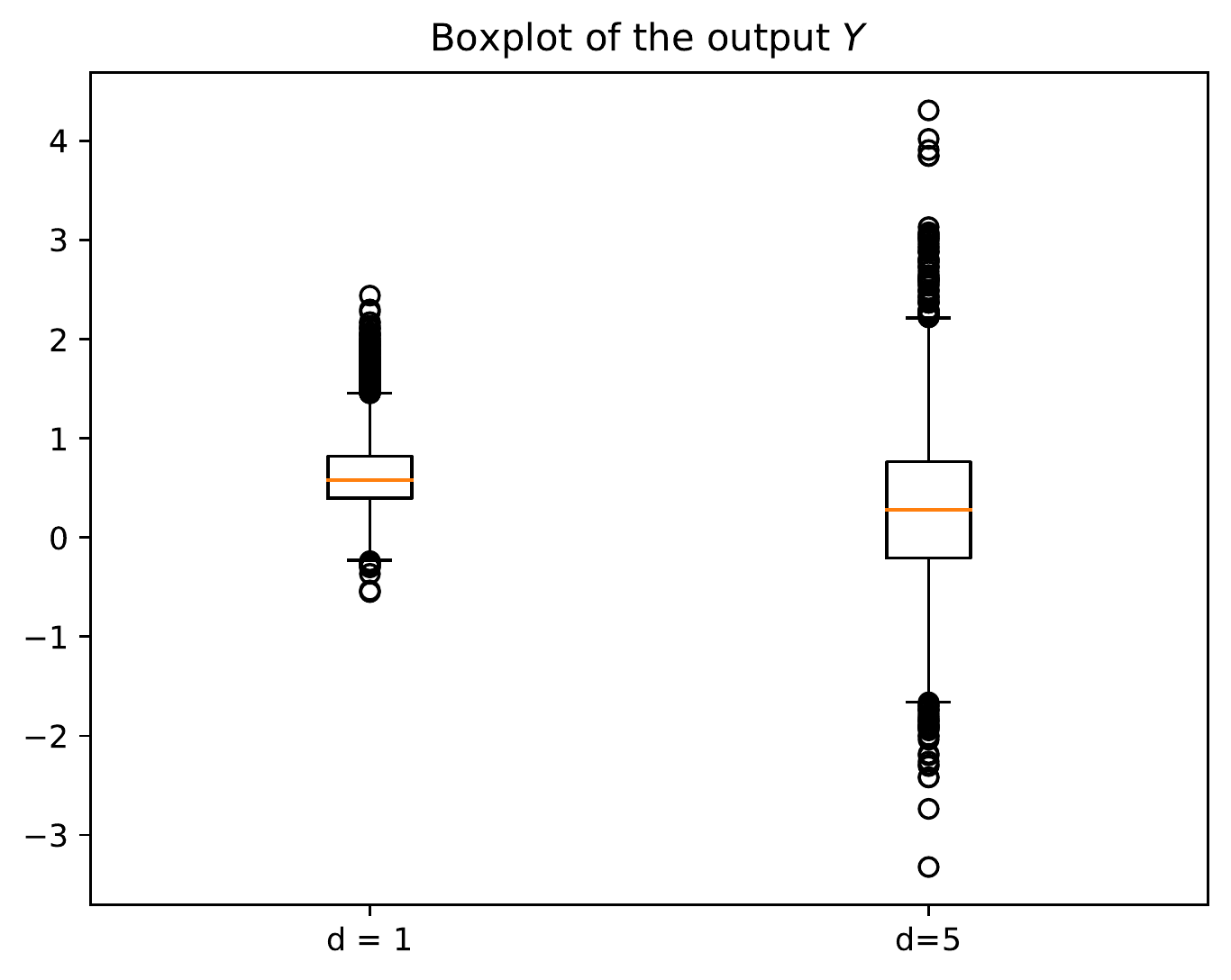}
    \caption{Boxplot of the output $Y$ for $d = 1,5$}
    \label{fig:boxplotModel}
\end{figure}
For $\ell \in \{0.1, 0.5, 1, 2\}$,  
we provide the evaluation of the expected length and the error rate for the oracle prediction set $\Gamma_{\ell}^*$. 
To this end, we repeat $100$ times the following scheme.
\begin{enumerate}
\item[i)] estimate $\lambda^*_{\ell}$ from an {\it unlabeled} dataset of size $N = 1000$ on a regular grid of size $M=1000$ of the interval $[-5,5]$;
\item[ii)] derive the resulting  prediction interval on the same grid over a test set of size $T =1000$;
\item[iii)] based on the test set, compute the expected length and the error rate.
\end{enumerate}
From these repetitions, we compute the mean and standard deviation of the estimates. The obtained results are provided in Table~\ref{tab:perfBayes}.

 \begin{table}[]
    \centering
    \begin{tabular}{c ||c | c || c | c} 
               & \multicolumn{2}{c||}{Expected length} &  \multicolumn{2}{c}{Error rate}\\ \hline
    $\ell$     & $d=1$ & $d=5$ & $d=1$ & $d=5$  \\\hline
    $0.1$      &  0.1 (0.01)    &0.1 (0.01) & 0.81 (0.01) & 0.94 (0.01)\\
    $0.5$      & 0.49 (0.01) & 0.49 (0.01) &  0.34 (0.01) & 0.71 (0.01) \\
    $1$     &  0.99 (0.01)   & 0.99 (0.01) & 0.07 (0.01) & 0.48 (0.01)\\
    $2$   & 1.99 (0.03) & 1.99 (0.01) & 0.00 (0.00) & 0.17 (0.01) \\
    \end{tabular}
    \caption{Performance of the Oracle PI for $\ell \in \{0.1,0.5,1,2\}$.}
    \label{tab:perfBayes}
\end{table}
 
\paragraph*{Simulation scheme.}

To assess the performance of our procedure,  we consider the following scheme. For $d \in \{1,5\}$ and $\ell \in \{0.1,0.5,1,2\}$,  we repeat $100$ the following steps.
\begin{enumerate}
    \item[i)] estimate $f^*$ and $\sigma^2$ from a training test of size $n=500$. We consider the residual-based method~\cite{Hall_Caroll89}. The estimation of $f^*$ and $\sigma^2$ relies on the random forests algorithm from \texttt{python} library \texttt{sklearn}. We also choose $u=10^{-5}$ for the parameter of the perturbation $\zeta$ (see Eq.~\eqref{eq:addpertubation});
    \item[ii)] compute $\hat{G}^{-1}(\ell)$ using an {\it unlabeled} dataset of size $N = 100$ on a regular grid of size $M=100$ of the interval $[-s,s]$, where $s = \max(-\min(Y_{train}); \max(Y_{train})$,
    \item[iii)] derive the resulting  prediction interval on a regular grid of size $1000$ of $[-s,s]$ over a test set of size $T =1000$;
    \item[iv)] based on the test set, compute the expected length and the error rate.
\end{enumerate}
From these experiments, we compute the empirical means and standard deviations expected length and the error rate. The results are provided in Table~\ref{tab:perfPred}. A visual description of the behavior of our PI is also given in Figure~\ref{fig:figPerf}.

Notice that the value of $s$ that we consider here is different from the one suggested by the theory in Theorem~\ref{thm:RatesOfcve}. This is a minor point. The parameter $s$ in the theory is set such that most of the labels lie in $[-s,s]$ with high probability. This happens when $n$ and $M$ grow since $s = \log(\min(n,N))$. Our choice in practice ensures that this property holds regardless the values of $n$ and $N$.


\begin{table}[]
    \centering
    \begin{tabular}{c ||c | c || c | c} 
               & \multicolumn{2}{c||}{Expected length} &  \multicolumn{2}{c}{Error rate}\\ \hline
    $\ell$     & $d=1$ & $d=5$ & $d=1$ & $d=5$  \\\hline
    $0.1$      &  0.1 (0.01)    &0.1 (0.02) & 0.81 (0.02) & 0.94 (0.01)\\
    $0.5$      & 0.50 (0.01) & 0.50 (0.02) &  0.34 (0.02) & 0.72 (0.02) \\
    $1$     &  1.00 (0.02)   & 1.00 (0.02) & 0.07 (0.01) & 0.48 (0.01)\\
    $2$   & 2.01 (0.06) & 2.01 (0.01) & 0.00 (0.00) & 0.17 (0.01) \\
    \end{tabular}
    \caption{Performance of $\hat{\Gamma}$ for $\ell \in \{0.1,0.5,1,2\}$.}
    \label{tab:perfPred}
\end{table}

\begin{figure}
    \centering
    \includegraphics[scale= 0.4]{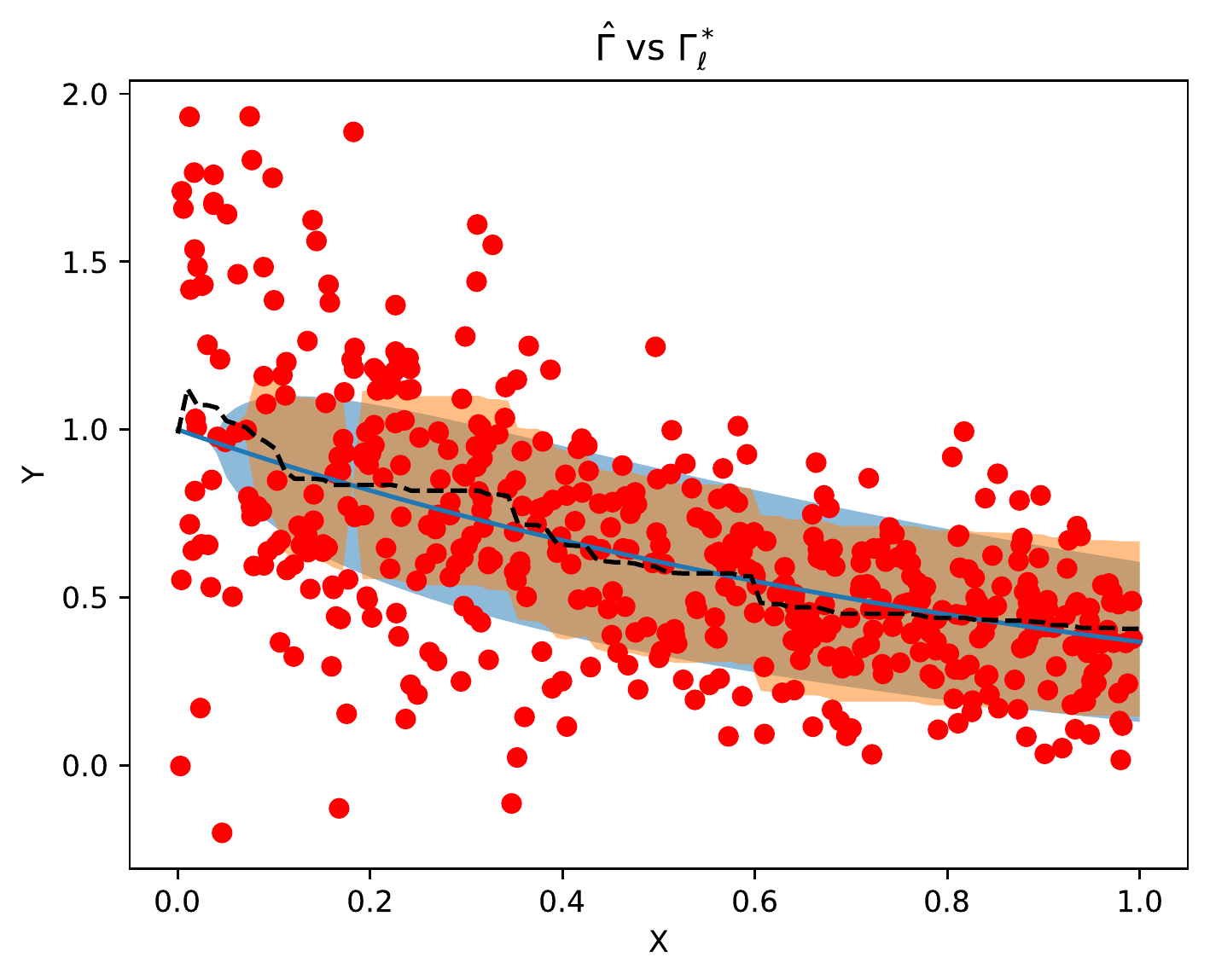} \hspace*{0.5cm}
    \includegraphics[scale= 0.4]{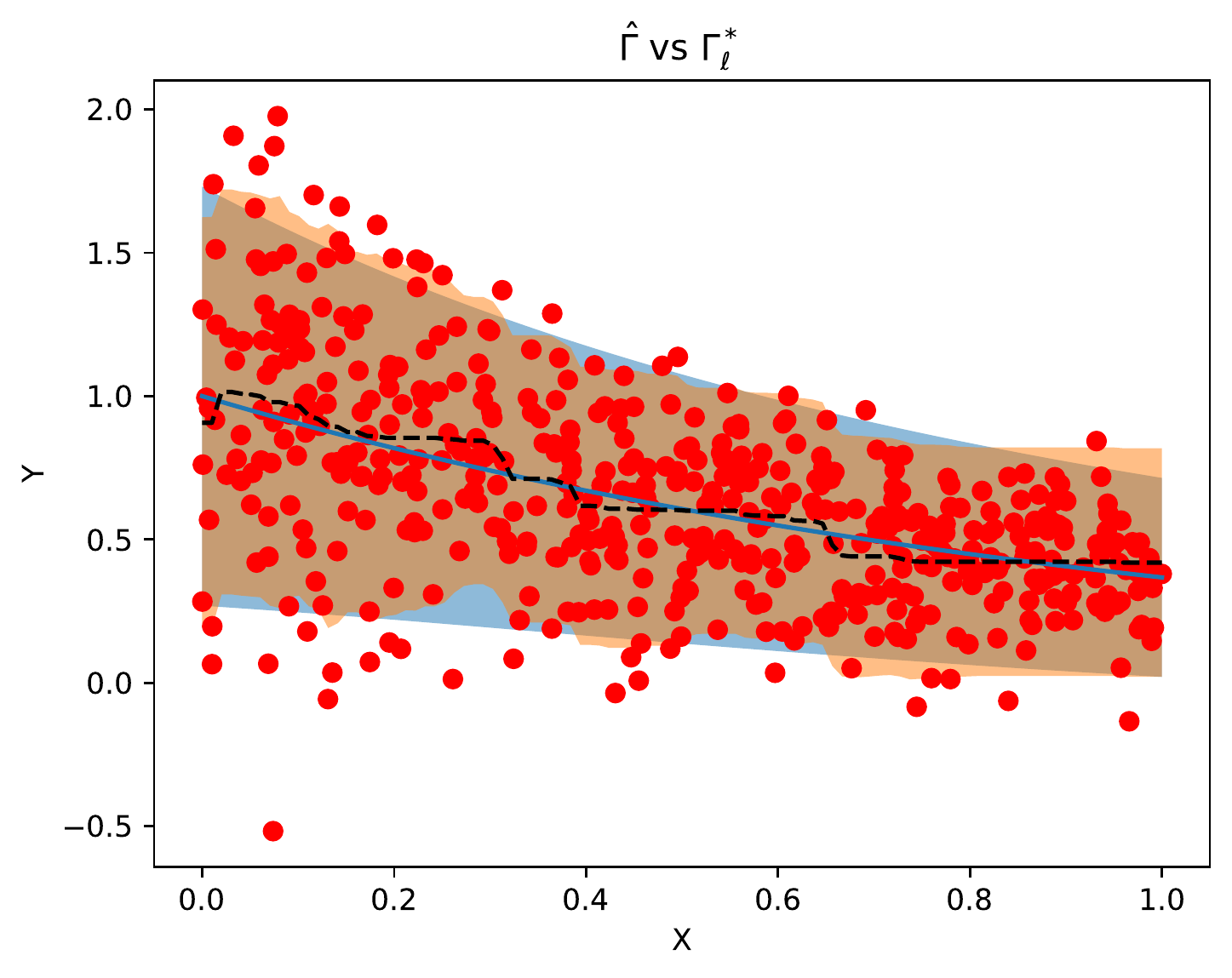}
    \caption{Visual description of the empirical PI $\hat{\Gamma}$ and its oracle counterpart $\Gamma^*_{\ell}$, with $\ell = 0.5$ on the left and $\ell=1$ on the right for $d = 1$. The scatter plot of data is displayed and the graph of both regression function $f^*$ and estimator $\hat{f}$ is represented (solid line for $f^*$, dashed line for $\hat{f}$. The oracle PI $\Gamma^*_{\ell}$ (empirical PI $\hat{\Gamma}$, respectively) is given in blue (orange, respectively).}
    \label{fig:figPerf}
\end{figure}

\paragraph*{Results.}
Two conclusions can be made from this first numerical study. First Tables~\ref{tab:perfBayes} and~\ref{tab:perfPred} highlight how effective our method is in producing PI with (almost) exactly the right length. This is an important point and suggests that our strategy succeeds to enforce the constraint on the length prescribed by the optimization problem.
Second, let us focus on a comparison between $\Gamma_{\ell}^*$, the oracle PI, and its empirical counterpart $\hat{\Gamma}$. Table~\ref{tab:perfBayes} and Table~\ref{tab:perfPred} show how close are the performance of these two PI both in terms of expected length and of error rate. Interestingly, the performance of $\hat{\Gamma}$ is obtained with a moderate size $N$ of the unlabeled sample that is used to estimate the threshold. These results also suggest that $n=500$ is enough to have good estimations of the regression and variance functions. The closeness between $\Gamma_{\ell}^*$ and $\hat{\Gamma}$ is also illustrated in Figure~\ref{fig:figPerf}.

\subsection{Numerical comparison with expected coverage approach}
\label{subsec:numCompaLei}
In this section, we numerically compare our procedure to the approach that constraint the expected coverage described in Section~\ref{subsec:conformal}. We consider the model defined in Equation~\ref{eq:modelSimu} with $d = 5$ and focus on the estimation of $\Gamma^*_{\ell}$ for $\ell = 2$. With this expected length, the oracle predictor $\Gamma^*_{\ell}$ reaches an error rate of $\beta = 0.17 $. Therefore, for this learning task, we are able to provide empirical PI for both approaches. That is to say, we compute $\hat{\Gamma}$ with  $\ell = 2$ as expected length and $ \hat{\Gamma}_{\beta}$ with $\beta = 0.17$ as expected error.
In order to get a fair comparison of the methods, we repeat $20$ times the following steps. 
For both approaches, we use a training set of size $n=500$ to estimate the density $p$ and we estimate the threshold of the considered procedure with a dataset of size $N \in \{10, 30, 50, 70, 100,150,200,500,1000\}$. Finally, we compute the expected length and error rate of both empirical PI
over a test set of size $T=1000$.
From these repetitions, we compute empirical means and standard deviations.
The results are displayed in Figure~\ref{fig:comparison}.

As expected, in average, both methods behaves similarly. However there are important differences in favor of our approach. First, the convergence of our method is much faster to the mean value both for the expected length and the error rate. We notice that $N=10 $ is already enough for our method while more than $500$ samples are needed for the method that focus on the coverage as constraint.
Second, it seems that our construction is much more stable, in particular for length calibration. It illustrates the efficiency of our procedure to build prediction interval with the right expected length.

The two approaches are definitively not comparable in terms of objectives. Indeed, if we are really focused on constraining the error rate, then the length constraint appears (at first sight) sub-optimal and vice versa if we ask for interpretable outputs. However, our numerical analysis clearly suggests that our methodology is more stable: it induces a procedure with a lower variance.

\begin{figure}
    \centering
    \includegraphics[scale= 0.4]{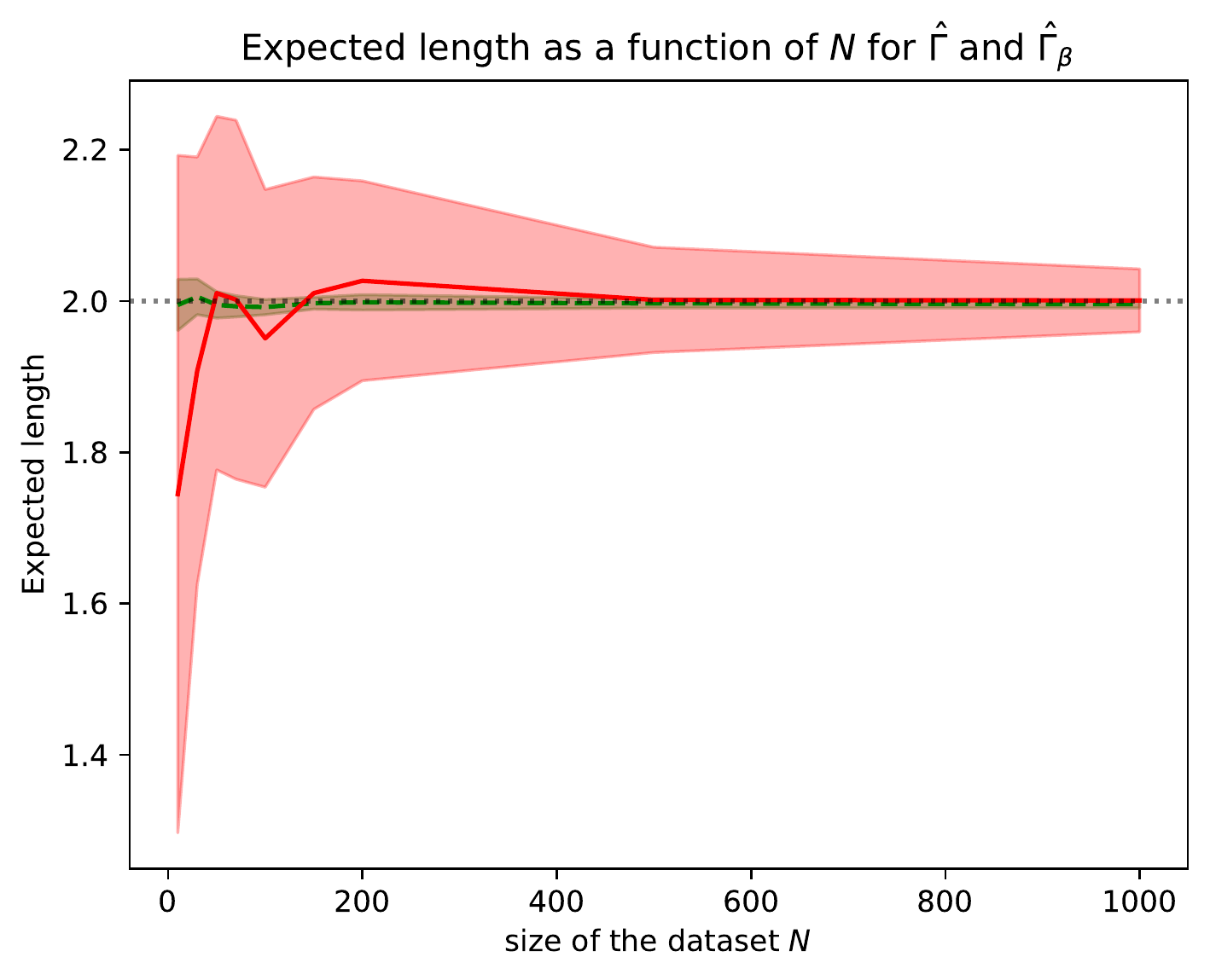} \hspace*{0.5cm}
    \includegraphics[scale= 0.4]{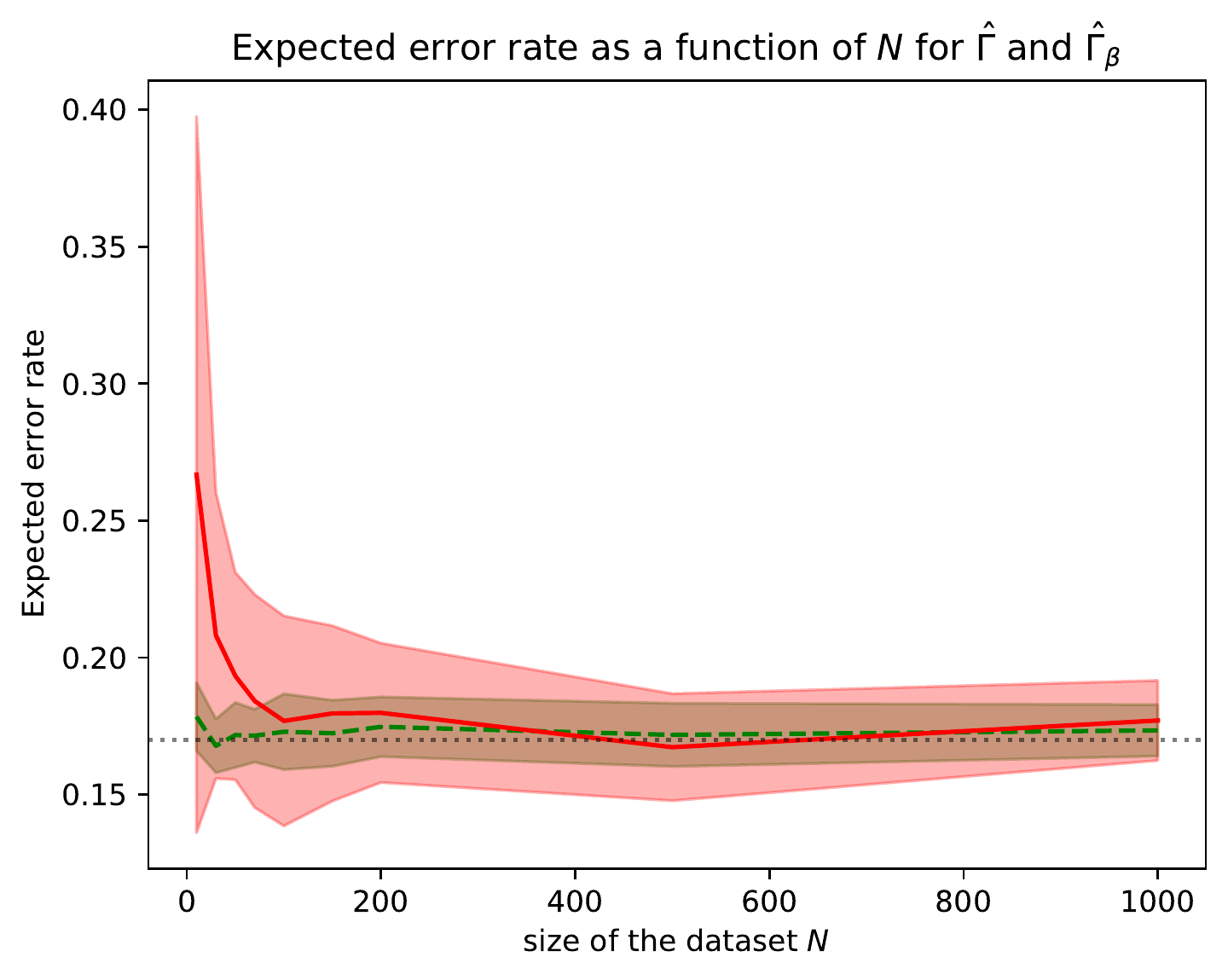}
    \caption{Comparison between $\hat{\Gamma}$ and $\hat{\Gamma}_{\beta}$. We plot the expected length (on the left) and the expected coverage (on the right) as a function of $N$ over $20$ repetitions for $\hat{\Gamma}$ (dashed) and $\hat{\Gamma}_{\beta}$ (solid line in red). The true value of the parameter is given by the dotted line.}
    \label{fig:comparison}
\end{figure}

\section{Conclusion}
\label{sec:conclusion}
 
In this paper, we provide a general methodology to build \emph{prediction intervals with controlled expected length} in the Gaussian regression. Our proposed algorithm is very effective in controlling the expected length of the output and then ensure the interpretability of the outcome. The theoretical analysis indicates that our method mimics the optimal rule \wrt the expected length and, under appropriate properties on the base estimators of the regression function, it is also efficient \wrt the symmetric difference distance and the excess-risk. Furthermore, a numerical study supports our theoretical results. Notably, it highlights good stability properties as compared to prediction intervals that focus on expected coverage constraints. 

Our numerical comparison to PI under expected coverage constraint additionally opens a very significant door to the use of our method. Because of the stability of our method, one may think to the following two-stage procedure to produce a PI with error rate $\beta$. 
\begin{itemize}
    \item{\it Step 1}. Build the PI with error rate $\beta$ and evaluate its length $\tilde{\ell}$;
    \item{\it Step 2}. Build our PI with average length $\tilde{\ell}$.
\end{itemize}
While we do not expect a significant improvement in average, the resulting prediction interval might be more stable. This will be the purpose of future investigation.

On the other hand, inference in the high-dimensional setting is a crucial challenge with modern data. Several successful studies consider the Gaussian \emph{homoscedastic} linear regression~\cite{Lu_inferenceLasso_17,Sara_TestHighdim_14,Minnier_PerturbInfe_11,Belloni_infer_14}. An important direction for future research is to carry out PI i) for non Gaussian models; ii) and that can handle heteroscedastic model. Both of these questions have their applications in the high dimensional setting.

\bibliographystyle{plain}
\bibliography{BIB}

\newpage
\appendix

\section*{Appendix}

This appendix is devoted to the proof of our main results.
The proofs related to Section~\ref{sec:framework} are provided in Section~\ref{sec:proofS2}, while Section~\ref{sec:proofS3} is devoted to the proofs of Section~\ref{sec:estimation}. Finally,  Section~\ref{sec:techRes} gathers useful results. In particular, we give rates of convergence for $K$NN estimates for both regression and variance function. Notice that in the whole appendix, $C$ is a positive constant that may change from one line to another.

\section{Technical results}
\label{sec:techRes}

In this section, we provide some useful properties that are used for the proof of our main results

\subsection{Technical lemmas}
The first tool we introduce is a generalization of the classical inverse transform theorem~\cite[Lemma 21.1]{vanderVaart98} to the continuous case. Let $a>0$. We consider a random process $(Z_y)_{y\in [-a,a]}$ such that the function $H$ defined by 
\begin{equation*}
H(t)=\frac{1}{2a}\int_{-a}^{a}\mathbb{P}(Z_y\geq t)dy \enspace,
\end{equation*}
is continuous on $\mathbb{R}_{+}$.
\begin{lemma}
\label{lem:UnifDist}
Let $T$ uniformly distributed on $[-a,a]$ and independent of $(Z_y)_{y\in [-a,a]}$. We consider the random variable $Z_T$ and let $U$ be distributed according to the uniform distribution on $[0,1]$. Then
\begin{eqnarray*}
H(Z_T)\overset{\mathcal{L}}=U \enspace \text{and} \enspace H^{-1}(U)\overset{\mathcal{L}}=Z_T \enspace.
\end{eqnarray*}
\end{lemma}
\begin{proof}
For every $t\geq 0$, we have $\mathbb{P}(H(Z_T)\leq t)= \mathbb{P}(Z_T\geq H^{-1}(t))$. Denote by $d\mathbb{P}_{T}$ the marginal distribution of $T$. Since the variable $T$ is independent of $(Z_y)_{y\in [-a,a]}$ and $H$ is continuous, one gets
\begin{eqnarray*}
\mathbb{P}(H(Z_T)\leq t)&=&\int \mathbb{P}(Z_T\geq H^{-1}(t)|T=y)) \ d\mathbb{P}_{T}(y)\\
&=& \frac{1}{2a}\int_{-a}^{a} \mathbb{P}(Z_y\geq H^{-1}(t)|T=y)) \ dy\\
&=&\frac{1}{2a}\int_{-a}^{a} \mathbb{P}(Z_y\geq H^{-1}(t)) \ dy=H(H^{-1}(t))=t \enspace,
\end{eqnarray*}
and we deduce that $H(Z_T)\overset{\mathcal{L}}=U$. For the second point of the Lemma, we observe that 
\begin{eqnarray*}
\mathbb{P}(H^{-1}(U)\leq t)&=& \mathbb{P}(U\geq H(t))=\frac{1}{2a}\int_{-a}^{a}\mathbb{P}(Z_y\leq t)dy= \frac{1}{2a}\int_{-a}^{a} \mathbb{P}(Z_y\leq t|T=y))dy=\mathbb{P}(Z_T\leq t)  \enspace .
\end{eqnarray*}
\end{proof}
\subsection{Rates of convergence for K-NN estimators}

In this section, we gather the results we use for $K$-NN estimators of both regression and variance function. The proof of this result is provided in \cite{Denis_Hebiri_Zaoui20}.

\begin{theorem}
\label{thm:cveKnn}
Grants Assumptions~\ref{ass:regularity}, ~\ref{ass:StrongDensityAssumption}, for
$k_n\propto n^{-2/d+2}$, and all $\alpha > 0$, the $K$-NN estimators defined in Equation~\eqref{eq:knnestimators} satisfy
\begin{eqnarray*}
\mathbb{E}\left[\left(\sup_{\bx \in \mathcal{C}}|\hat{f}(\bx) - f^{*}(\bx)|\right)^{1+\alpha}\right]  & \leq &  C\log(n)^{1+\alpha}n^{-(1+\alpha)/(2+d)}\enspace,\\   
\mathbb{E}\left[\left(\sup_{\bx \in \mathcal{C}}|\hat{\sigma}^2(\bx) - \sigma(\bx)|\right)^{1+\alpha}\right]  & \leq &  C\log(n)^{1+\alpha}n^{-(1+\alpha)/(2+d)}\enspace.
\end{eqnarray*}
\end{theorem}

\section{Proof of Section~\ref{sec:framework}}
\label{sec:proofS2}

In this section, we provide proofs related to the optimal confidence and to the excess-risk formula

\begin{proof}[Proof of Proposition~\ref{prop:oraclePredictor}] First, let us consider the Lagrangian of the optimization problem~\ref{eq:eqOracle}. It can be written as
$$
H(\Gamma,\lambda)= \mathbb{P}\left(Y\notin \Gamma (\bX)\right) + \lambda \left(\mathbb{E}_{\bX}[L(\Gamma(\bX)]-\ell\right)\enspace ,
$$
where $\lambda \geq 0$ is a dual variable of the problem. Since, $$
\mathbb{P}\left(Y\in \Gamma (\bX)\right) = \mathbb{E}_{\bX}\left[\mathbb{E}\left[\one_{\{Y\in\Gamma (\bX)\}}|\bX\right]\right] = \mathbb{E}_{\bX}\left[\int_{\bbR} p(y|\bX)\one_{\{y\in\Gamma (\bX)\}}dy\right]\enspace,
$$
the Lagrangian reads as
\begin{equation}
    \label{eq:proofLagrangian}
    H(\Gamma,\lambda)= 1 - \lambda \ell - \mathbb{E}_{\bX}\left[\int_{\bbR}(p(y|\bX)-\lambda)\one_{\{y\in\Gamma (\bX)\}}dy\right] \enspace.
\end{equation}
Minimizing \emph{w.r.t.} $\Gamma$ leads to an optimal solution that can be written for all $\lambda\geq 0$ and all $\bx \in \bbR^d$ as
$$
\Gamma^*(\lambda,\bx) = \left\{y\in \mathbb{R}: p(y|\bX)\geq \lambda\right\}\enspace .
$$
Injecting this value into~\ref{eq:proofLagrangian} gives
$$
H(\Gamma^*(\lambda,\bX),\lambda)= 1 - \lambda \ell - \mathbb{E}_{\bX}\left[\int_{\bbR}(p(y|\bX)-\lambda)_+\one_{\{y\in\Gamma^* (\lambda,\bX)\}}dy\right] \enspace,
$$
where $(\cdot)_+$ stands for the positive part. First order optimality conditions for convex non-smooth minimization
problems implies $0 \in \partial H(\Gamma^*(\lambda_{\ell}^*,\bX),\lambda_{\ell}^*)$ where $\partial H$ is the sub-differential of $H$. Therefore, using the Fundamental Theorem of Calculus, we get 
$\mathbb{E}_{\bX} \left[\int_{\bbR} \one_{\{y\in\Gamma^* (\lambda_{\ell}^*,\bX)\}}dy\right] = \ell .$
But, using the above definition of $\Gamma^*$ we can write by  Fubini's theorem the left hand side term as $\mathbb{E}_{\bX} \left[\int_{\bbR} \one_{\{y\in\Gamma^* (\lambda_{\ell}^*,\bX)\}}dy\right] = \int_{\bbR} \mathbb{P}\left( (p(y|\bX) \geq \lambda_{\ell}^* \right) dy = G(\lambda_{\ell}^*)$. We then conclude that $\lambda_{\ell}^* = G^{-1}(\ell)$. Notice that for this value, we have 
\begin{eqnarray*}
\mathcal{L}(\Gamma^*) = \mathbb{E}_{\bX}[L(\Gamma^{*}(\lambda_{\ell}^*, \bX)] = \mathbb{E}_{\bX}\left[\int\one_{\{y\in\Gamma^{*}(\lambda_{\ell}^*, \bX) \}} dy\right]= G(\lambda_{\ell}^*)=\ell \enspace . 
\end{eqnarray*}
\end{proof}

\begin{proof}[Proof of Proposition~\ref{prop:propExcessRisk}]
Let $\ell\geq 0$. Considering a similar decomposition as in the proof of Proposition~\ref{prop:oraclePredictor}, we can write the error rate of a predictor $\Gamma$ as
\begin{equation}
\label{eq:ProofexcessRisk}
R_{\ell}(\Gamma)=1 -\mathbb{E}_{\bX}\left[\int_\bbR (p(y|\bX)-\lambda_{\ell}^*)\one_{\{y\in\Gamma (\bX)\}}dy\right] \enspace .
\end{equation}
Therefore, we deduce
$$
\mathcal{E}_{\ell}\left(\Gamma\right) = \mathbb{E}_{\bX}\left[\int_\bbR (p(y|\bX) - \lambda_{\ell}^* )\left( \one_{\{y\in\Gamma^{*}_{\ell}(\bX)\}} -   \one_{\{y\in\Gamma (\bX)\}} \right)dy\right]\enspace,
$$
and the result follows from the fact that $\one_{\{y\in\Gamma^{*}_{\ell}(\bX)\}} -  \one_{\{y\in\Gamma (\bX)\}} = \sgn(p(y|\bX) - \lambda_{\ell}^* )$ since we have the equality between events $\{ y\in \Gamma^{*}_{\ell}(\bX) \} = \{ p(y|\bX) - \lambda_{\ell}^* \geq 0\}$, where $\sgn:\bbR\to\{-1,1\}$ stands for the sign.
\end{proof}

\section{Proof of Section~\ref{sec:estimation}}
\label{sec:proofS3}
We now consider the theoretical properties of the prediction interval $\hat{\Gamma}$. We first consider its expected length and then derive a finite sample bound on its excess-risk.
\subsection{Length control}
\begin{proof}[Proof of Proposition~\ref{prop:propLength}]
To show this result, we need to introduce some pseudo-oracle predictor that has expected length $\ell$. Let us then define the randomized predictor

\begin{equation}
\label{eq:pseudoOracle}
\bar{\Gamma}(\bX,\zeta) = \{y \in \mathbb{R}: \ \hat{p}(y|\bX,\zeta) \geq \bar{G}^{-1}(\ell)\}   \enspace , 
\end{equation}
where $\bar{G}(t):=\int_{\mathbb{R}} \mathbb{P}_{\bX,\zeta}(\hat{p}(y|\bX,\zeta )\geq t)dy$ for all $t>0$. Here again, the property $\mathcal{L}(\bar{\Gamma}) := \mathbb{E}_{\bX,\zeta}\left[L\left(\bar{\Gamma}(\bX,\zeta)\right)\right] = \ell$ is due to the fact that the conditional on the data $\mathcal{D}_{n}$ the r.v. $\hat{p}(y|\bX,\zeta) $ has no atoms since it is randomized.

Let us now consider the purpose of the proposition. We need to bound $\mathbb{E}\left[ \big|  \mathcal{L}(\hat{\Gamma}) -\ell \big| \right]$. 
We can write
\begin{eqnarray}
\label{eq:decompSize}
\big|  \mathcal{L}(\hat{\Gamma}) -\ell \big|  = \big|\mathcal{L}(\hat{\Gamma})  - \mathcal{L}(\bar{\Gamma})  \big| 
&=&\bigg|\mathbb{E}\left[\int_{\mathbb{R}}\left(\one_{\{\hat{G}(\hat{p}(y|\bX,\zeta))\leq \ell\}}-\one_{\{\bar{G}(\hat{p}(y|\bX,\zeta)) \leq \ell\}}\right)dy\right]\bigg|\\ \nonumber
&\leq & \mathbb{E}\left[\int_{\mathbb{R}}\bigg|\one_{\{\hat{G}(\hat{p}(y|\bX,\zeta)) \leq \ell\}}-\one_{\{\bar{G}(\hat{p}(y|\bX,\zeta)) \leq \ell\}}\bigg|dy\right]\\ \nonumber
&\leq &\mathbb{E}\left[\int_{\mathbb{R}}\one_{\{|\hat{G}(\hat{p}(y|\bX,\zeta))-\bar{G}(\hat{p}(y|\bX,\zeta))|\geq |\bar{G}(\hat{p}(y|\bX,\zeta))-\ell|\}}dy\right]\\ \nonumber
&=& \int_{\mathbb{R}}\mathbb{P}\left(|\hat{G}(\hat{p}(y|\bX,\zeta))-\bar{G}(\hat{p}(y|\bX,\zeta))|\geq |\bar{G}(\hat{p}(y|\bX,\zeta))-\ell|\right)dy \enspace ,
\end{eqnarray}
where we use Fubini's theorem at last. Now notice that the above integral is limited to the compact $[-s,s]$ since, this is the support of the function $\hat{p}(\cdot|\bx,z) $ for all $(\bx,z)\in \bbR^d \times [0,u]$.
To bound this integral, we make use of the peeling technique of~\cite{Audibert_Tsybakov07}. That is, we consider for $\delta>0$ and $y\in [-s,s]$
\begin{eqnarray*}
A_{0}(y)&=& \left\{0\leq |\bar{G} (\hat{p}(y|\bX, \zeta)) -\ell|\leq \delta\right\}\\
A_{j}(y)&=& \left\{2^{j-1}\delta\leq |\bar{G}(\hat{p}(y|\bX, \zeta))-\ell|\leq 2^{j}\delta\right\}, \qquad \text{for} \quad j\geq 1 \enspace .
\end{eqnarray*}
Since for $y\in [-s,s]$, the events $(A_{j}(y))_{j\geq 0}$ are mutually exclusive, we deduce
\begin{multline}
\label{eq:IntPA(y)j1}
\int_{-s}^{s}\mathbb{P}\left(|\hat{{G}} (\hat{p}(y|\bX, \zeta)) -\bar{G}(\hat{p}(y|\bX, \zeta))|\geq |\bar{G}(\hat{p}(y|\bX, \zeta))-\ell|\right)\ dy=\\
\int_{-s}^{s}\sum_{j\geq 0}\mathbb{P}\left(|\hat{{G}} (\hat{p}(y|\bX, \zeta))-\bar{G}(\hat{p}(y|\bX, \zeta))|\geq |\bar{G}(\hat{p}(y|\bX, \zeta))-\ell| \ ,\  A_{j}(y)\right) \ dy \enspace .
\end{multline}
Controlling this term relies on a bound on $\int_{-s}^{s}\mathbb{P}(A_{j}(y))dy  $. 
It is clear that $
0\leq \bar{G}(t)=\int_{-s}^{s} \mathbb{P}_{X}(\hat{p}(y|\bX, \zeta)\geq t|\mathcal{D}_n)dy \leq 2 s $ for all $t\in [0,1]$.
We can apply Lemma~\ref{lem:UnifDist} to say that $\bar{G}(Z_T)$ is uniformly distributed on $[0 , 2s]$ and then, 
for all $j\geq 0$ and $\delta>0$, we deduce that 
\begin{multline}
\label{eq:IntPA(y)j2}
\int_{-s}^{s}\mathbb{P}(A_{j}(y))dy   =  2s \frac{1}{2s}\int_{-s}^{s}\mathbb{P}\left(|\bar{G}(\hat{p}(y|\bX,\zeta))-\ell|
 \leq 2^j\delta \ | \ \mathcal{D}_n\right)dy \\ 
  =  2s \times \mathbb{P}\left(|\bar{G}(Z_T)-\ell | \leq 2^j\delta|\mathcal{D}_n\right)\leq 2s \frac{2^{j+1}\delta}{2s} = 2^{j+1}\delta \enspace .
\end{multline}
Next, let us consider~\eqref{eq:IntPA(y)j1}. We observe that for all $j\geq 1$
\begin{multline}
\label{eq:IntPA(y)j3}
\int_{-s}^{s}\mathbb{P}\left(|\hat{{G}}(\hat{p}(y|\bX,\zeta))-\bar{G}(\hat{p}(y|\bX,\zeta))|\geq |\bar{G}(\hat{p}(y|\bX,\zeta))-\ell| \ , \  A_{j}(y)\right)dy \\ 
\leq 
\int_{-s}^{s}\mathbb{P}\left(|\hat{{G}}(\hat{p}(y|\bX,\zeta))-\bar{G}(\hat{p}(y|\bX,\zeta))|\geq 2^{j-1}\delta \ 
,  \ A_{j}(y)\right)dy  \\
\leq
\int_{-s}^{s}\mathbb{E}_{(\mathcal{D}_n,\bX,\zeta)}\left[\mathbb{P}_{\mathcal{D}_N}\left(|\hat{{G}}(\hat{p}(y|\bX,\zeta))-\bar{G}(\hat{p}(y|\bX,\zeta))|\geq 2^{j-1}\delta \right)    \one_{ A_{j}(y)}\right]dy.
\end{multline}
In Section~\ref{subsec:dataDrivenProce}, we have presented the predictor $\hat{\Gamma}$ that relies on the function $\hat{G}$ which is discretized. On the other hand, $\bar{G}$ is not discretized. Because of this difference, it is convenient, in order to control~\eqref{eq:IntPA(y)j3}, to provide some additional notation. Let us define 
\begin{equation*}
\hat{\bar{G}}(t):= \frac{1}{N} \sum_{i=1}^N \int_{-s}^s \one_{ \{   \hat{p}(y|\bX_{n+i}, \zeta_i )  \geq t \} } dy \enspace .
\end{equation*}
Then for all $y\in [-s,s]$, conditional on $(\mathcal{D}_{n},\bX,\zeta)$, the probability in Eq.~\eqref{eq:IntPA(y)j3} is bounded as follows
\begin{eqnarray}
\label{eqproof:triangleG}
\mathbb{P}_{\mathcal{D}_N}\left(|\hat{{G}}(\hat{p}(y|\bX,\zeta))-\bar{G}(\hat{p}(y|\bX,\zeta))|\geq 2^{j-1}\delta\right)  \leq & \nonumber \\
\mathbb{P}_{\mathcal{D}_N}\left(|\hat{\bar{G}}(\hat{p}(y|\bX,\zeta))-\bar{G}(\hat{p}(y|\bX,\zeta))|\geq 2^{j-1}\frac{\delta}{2}\right)
+ & \mathbb{P}_{\mathcal{D}_N}\left(|\hat{\bar{G}}(\hat{p}(y|\bX,\zeta))-\hat{G}(\hat{p}(y|\bX,\zeta))|\geq 2^{j-1}\frac{\delta}{2}\right) \enspace.
\end{eqnarray}
These two last terms are treated in different ways. For the first one, we observe that for all $t\in[0,1]$
\begin{eqnarray*}
    |\hat{G}(t)-\hat{\bar{G}}(t)| & = &  \left|  \frac{1}{N} \sum_{i=1}^N  \sum_{k=1}^M \left( \int_{y_k}^{y_{k+1}} \one_{\{\hat{p}(y|\bX_{n+i}, \zeta_i) \geq t   \}}  -    \one_{ \{   \hat{p}(y_{k}|\bX_{n+i}, \zeta_i )  \geq t \} }   \right)  dy \right| 
    \\
    & \leq & \frac{1}{N} \sum_{i=1}^N  \sum_{k=1}^M \left( \int_{y_k}^{y_{k+1}} \left| \one_{\{\hat{p}(y|\bX_{n+i}, \zeta_i) \geq t   \}}  -    \one_{ \{   \hat{p}(y_{k}|\bX_{n+i}, \zeta_i )  \geq t \} }  \right|  \right)  dy \enspace.
\end{eqnarray*}
We recall that for all $|y| \leq s $, we have $\hat{p}(y|\bx, \zeta ) = \hat{p}(y|\bx) + \zeta  $.
Because, conditional on $\mathcal{D}_n$, the function $\hat{p}(\cdot|\bx )$ is a Gaussian density and since the perturbation $\zeta $ acts on each $y$ in the same way, it turns out that the function $\hat{p}(.|\bx, \zeta )$ is continuously increasing and then decreasing with a maximum at $y=\hat{f}(\bx)$. Therefore, for any fixed $t$ the indicators $\one_{\{\hat{p}(y|\bX_{n+i}, \zeta_i) \geq t   \}} $ and $\one_{ \{   \hat{p}(y_{k}|\bX_{n+i}, \zeta_i )  \geq t \} } $ differ at most in $2$ intervals of the form $[y_k,y_{k+1}]$. Then we deduce that
\begin{equation*}
|\hat{G}(t)-\hat{\bar{G}}(t)| \leq 2 \times \frac{2s}{M} \enspace.
\end{equation*}
Injecting this inequality to~\eqref{eqproof:triangleG} gives
\begin{multline}
\label{eq:FinalDecomp}
\mathbb{P}_{\mathcal{D}_N}\left(|\hat{{G}}(\hat{p}(y|\bX,\zeta))-\bar{G}(\hat{p}(y|\bX,\zeta))|\geq 2^{j-1}\delta\right)  \leq \\
\mathbb{P}_{\mathcal{D}_N}\left(|\hat{\bar{G}}(\hat{p}(y|\bX,\zeta))-\bar{G}(\hat{p}(y|\bX,\zeta))|\geq 2^{j-1}\frac{\delta}{2}\right)
+ \one_{\{4s/M \geq 2^{j-2}\delta\}}\enspace.
\end{multline}
Let us now consider the second term. Conditional on $(\mathcal{D}_{n},\bX,\zeta)$, the random variable $\hat{\bar{G}}(\hat{p}(y|\bX,\zeta))$ is an empirical mean of \iid random variables of common mean $\bar{G}(\hat{p}(y|X,\zeta))\in [0,2s]$, we deduce from Hoeffding's inequality that 
\begin{eqnarray*}
\mathbb{P}_{\mathcal{D}_N}\left(|\hat{\bar{G}}(\hat{p}(y|\bX,\zeta))-\bar{G}(\hat{p}(y|\bX,\zeta))|\geq 2^{j-2}\delta|\mathcal{D}_{n},\bX\right)\leq 2\exp\left(\frac{-N\delta^{2}2^{2j-1}}{16s^2}\right) \enspace .
\end{eqnarray*}
Therefore, from Inequalities~\eqref{eq:IntPA(y)j1},~\eqref{eq:IntPA(y)j2},~\eqref{eq:IntPA(y)j3}, and~\eqref{eq:FinalDecomp} one gets for $\delta= \frac{4s}{\sqrt{N}}$ and $M > 4\sqrt{N}$
\begin{multline}
\label{eq:sizeProbRate}
\int_{-s}^{s}\mathbb{P}\left(|\hat{G}(\hat{p}(y|\bX,\zeta))-\bar{G}(\hat{p}(y|\bX,\zeta))|\geq |\bar{G}(\hat{p}(y|\bX,\zeta))-\ell|\right)dy \\  \leq\int_{-s}^{s} \mathbb{P}(A_{0}(y))dy + \sum_{j\geq 1}2\exp\left(\frac{-N\delta^{2}2^{2j-1}}{16s^2}\right)\int_{-s}^{s}\mathbb{P}(A_{j}(y))dy \\ 
\leq 2\delta +\delta \sum_{j\geq 1}2^{j+2}\exp\left(\frac{-N\delta^{2}2^{2j-1}}{16s^2}\right)
\leq \dfrac{Cs}{\sqrt{N}}\enspace.
\end{multline}
\end{proof}

\subsection{Excess-risk control}

\begin{proof}[Proof of Proposition~\ref{prop:propConsist}]

Throughout the proof, we denote $\bar{\lambda}_{\ell} := \bar{G}^{-1}(\ell)$,
where $\bar{G}$ is defined in Equation~\eqref{eq:pseudoOracle}.
We start with the following decomposition.
\begin{equation}
\label{eq:excessRiskFirstdecomp}
\mathcal{E}_{\ell}\left(\hat{\Gamma}\right) = \mathcal{E}\left(\bar{\Gamma}\right)
+ \left(R_{\ell}(\hat{\Gamma}) - R_{\ell}(\bar{\Gamma})\right) \enspace.
\end{equation}
For the second term of the {\it r.h.s.} in the above equation, thanks to Equation~\eqref{eq:ProofexcessRisk}, we have that
\begin{equation*}
R_{\ell}(\hat{\Gamma}) - R_{\ell}(\bar{\Gamma}) = \mathbb{E}_{{\bf X},\zeta} \left[\int_{\mathbb{R}}\left(p(y|\bf X)-\lambda_{\ell} \right) \left( \one_{\{y\in \bar{\Gamma}({\bf X},\zeta)\}} - \one_{\{y\in \hat{\Gamma}({\bf X},\zeta)\} } \right) dy\right] \enspace. 
\end{equation*}
From Assumption~\ref{ass:assSigma}, we have that $\left|p(y|\bf X)-\lambda_{\ell} \right|$ is bounded by $C_1 > 0$ which depends on $\sigma_0$. Hence, we deduce that
\begin{equation*}
\mathbb{E}\left[ \left| R_{\ell}(\hat{\Gamma}) - R_{\ell}(\bar{\Gamma}) \right| \right]
 \leq C_1 \mathbb{E}\left[\int_\mathbb{R} \left| \one_{\{y\in \bar{\Gamma}({\bf X},\zeta)\}} - \one_{\{y\in \hat{\Gamma}({\bf X},\zeta)\}}  \right| dy \right]\enspace.
\end{equation*}
This last inequality can be rewritten as
\begin{equation*}
\mathbb{E}\left[\left| R_{\ell}(\hat{\Gamma}) - R_{\ell}(\bar{\Gamma}) \right| \right]
 \leq C_1 \mathbb{E}\left[\int_{\mathbb{R}}\bigg|\one_{\{\hat{G}(\hat{p}(y|\bX,\zeta)) \leq \ell\}}-\one_{\{\bar{G}(\hat{p}(y|\bX,\zeta)) \leq \ell\}}\bigg|dy\right]\enspace .
\end{equation*}
Therefore, from Equation~\ref{eq:decompSize}, and~\eqref{eq:sizeProbRate}, we deduce
\begin{equation}
\label{eq:diffRiskR}
\mathbb{E}\left[\left| R_l(\hat{\Gamma}) - R_l(\bar{\Gamma}) \right| \right]
 \leq C \dfrac{s}{\sqrt{N}} \enspace .
\end{equation}
Now we bound the first term in the {\it r.h.s.} in Equation~\eqref{eq:excessRiskFirstdecomp}.
Thanks to Proposition~\ref{prop:propExcessRisk}, we have that
\begin{equation*}
\mathcal{E}_{\ell}(\bar{\Gamma})= \mathbb{E}_{{\bf X},\zeta}\left[\int_{\bar{\Gamma}(\bX, \zeta)\triangle \Gamma_{\ell}^{*}(\bX)} \left|p(y|\bX)-\lambda^*_\ell \right| \ dy \right]\enspace .
\end{equation*}
Now, we consider the following cases
\begin{itemize}
     \item[$\bullet$]
     If $y\in \bar{\Gamma}(\bX, \zeta) \setminus \Gamma_{\ell}^{*}(\bX)$, we have that $p(y|\bX)<\lambda^*_\ell$ and $\hat{p}(y|\bX,\zeta)\geq \bar{\lambda}_\ell$.
     Therefore, 
\begin{equation*}
|p(y|\bX)-\lambda^*_{\ell}|=(\lambda^*_{\ell}- \bar{\lambda}_\ell)+(\bar{\lambda}_\ell-\hat{p}(y|\bX,\zeta))+(\hat{p}(y|\bX,\zeta)-p(y|\bX))\enspace.
\end{equation*}
Using the fact that $\bar{\lambda}_\ell- \hat{p}(y|\bX,\zeta)\leq 0$, we get
\begin{equation*}
\int |p(y|\bX)-\lambda^*_{\ell}|\one_{\{y \in \bar{\Gamma}(\bX, \zeta) \setminus \Gamma_{\ell}^{*}(\bX) \}}dy\leq \int\left( ( \lambda^*_{\ell}- \bar{\lambda}_\ell) +\big|\hat{p}(y|\bX,\zeta)-p(y|\bX)\big|\right)\one_{\{y\in \bar{\Gamma}(\bX, \zeta) \setminus \Gamma_{\ell}^{*}(\bX)\}}dy\enspace.
\end{equation*}
 \item[$\bullet$]
     If $y\in \Gamma_{\ell}^{*}(\bX) \setminus  \bar{\Gamma}(\bX, \zeta)$, we have that $p(y|\bX)\geq \lambda^*_\ell$ and $\hat{p}(y|\bX,\zeta)< \bar{\lambda}_\ell$. Therefore, 
\begin{equation*}
|p(y|\bX)-\lambda^*_{\ell}|=(p(y|\bX)-\hat{p}(y|\bX,\zeta))+(\hat{p}(y|\bX,\zeta)-\bar{\lambda}_\ell)+(\bar{\lambda}_\ell-\lambda^*_{\ell}) \enspace .
\end{equation*}
Using the fact that $ \hat{p}(y|\bX,\zeta)-\bar{\lambda}_\ell < 0$, we get
\begin{equation*}
\int |p(y|\bX)-\lambda_{\ell}|\one_{\{y\in \Gamma_{\ell}^{*}(\bX) \setminus  \bar{\Gamma}(\bX, \zeta)\}} dy\leq \int\left((\bar{\lambda}_\ell-\lambda^*_{\ell}) +\big|\hat{p}(y|\bX,\zeta)-p(y|\bX)\big|\right)\one_{\{y\in \Gamma_{\ell}^{*}(\bX) \setminus  \bar{\Gamma}(\bX, \zeta) \}}dy \enspace.
\end{equation*}
\end{itemize}
From the above considerations, we deduce the following inequality
\begin{multline*}
\mathbb{E}_{{\bf X},\zeta}\left[\int_{\bar{\Gamma}(\bX, \zeta)\triangle \Gamma_{\ell}^{*}(\bX)} \left|p(y|\bX)-\lambda^*_\ell \right| \ dy \right] 
\\
\leq 
| \bar{\lambda}_\ell-\lambda^*_{\ell} | \  \mathbb{E}\left[\int_{\mathbb{R}}  \one_{\{y\in {\bar{\Gamma}(\bX, \zeta)\triangle \Gamma_{\ell}^{*}(\bX)} \}} {\rm d}y\right]
+
\mathbb{E}\left[\int_{\mathbb{R}}\left|\hat{p}(y|\bX,\zeta)-p(y|\bX)\right|{\rm d}y\right] 
\\
\leq 
| \bar{\lambda}_\ell-\lambda^*_{\ell} | \times 
(\mathcal{L}(\bar{\Gamma})  - \mathcal{L}(\Gamma^*))  
+
\mathbb{E}\left[\int_{\mathbb{R}}\left|\hat{p}(y|\bX)-p(y|\bX)\right|{\rm d}y\right] + 2su \enspace ,
\end{multline*}
where the last inequality is due to the fact that $\hat{p}(y|\bX,\zeta) = \hat{p}(y|\bX) + \zeta \one_{y\in[-s,s]}$ with $ |\zeta| \leq u$. But $\mathcal{L}(\bar{\Gamma})   = \mathcal{L}(\Gamma^*) = \ell $ by construction. Then,
$$
\mathbb{E}\left[\mathcal{E}_{\ell}(\bar{\Gamma})\right] \leq 
\mathbb{E}\left[\int_{\mathbb{R}}\left|\hat{p}(y|\bX)-p(y|\bX)\right|{\rm d}y\right] + 2su \enspace .
$$
Injecting this last inequality and~\eqref{eq:diffRiskR} into ~\eqref{eq:excessRiskFirstdecomp} gives the announced result.
\end{proof}

\subsection{Consistency Result}

This section is devoted to the proof of Theorem~\ref{thm:consistance}. We first provide a result on the $L_1$-integrated estimation error of $\hat{p}$.
\begin{proposition}
\label{prop:propEstim}
Under Assumption~\ref{ass:assSigma}, we have that 
\begin{multline*}
\mathbb{E}\left[\int_{\mathbb{R}}\left|\hat{p}(y|\bX)-p(y|\bX)\right|{\rm d}y\right]
 \leq \\  
 C\left(\sqrt{s}\mathbb{E}\left[(\hat{f}(\bX)-f^*(\bX))^2\right] + \mathbb{E}\left[\left|\hat{f}(\bX)-f^*(\bX)\right|\right]\right)\\  + Cs^{5/2}\mathbb{E}\left[|\hat{\sigma}^2(\bX)-\sigma^2(\bX)|\right] \enspace ,
\end{multline*}
where $C>0$ is a constant which depends on $\sigma_0$ and $\sigma_1$ in Assumption~\ref{ass:assSigma}.
\end{proposition}
\begin{proof} 
To build this proof, we use the triangle inequality to split the term $\left|\hat{p}(y|\bX)-p(y|\bX)\right|$ into $3$. We then have to consider each of these terms consecutively.
The first of these terms can be bounded as follows:
\begin{multline}
\label{eq:eq1bis}
\left|\dfrac{1}{\sqrt{2\pi\hat{\sigma}^2(\bX)}} \exp\left(-\dfrac{(y-\hat{f}(\bX))^2}{2\hat{\sigma}^2(\bX)}\right)   -  \dfrac{1}{\sqrt{2\pi{\sigma^2(\bX)}}} \exp\left(-\dfrac{(y-\hat{f}(\bX))^2}{2\hat{\sigma}^2(\bX)}\right)\right| 
\\
\leq  
\dfrac{1}{\sqrt{2\pi\hat{\sigma}^2(\bX)}} \exp\left(-\dfrac{(y-\hat{f}(\bX))^2}{2\hat{\sigma}^2(\bX)}\right) \left|1-\dfrac{\hat{\sigma}(\bX)}{\sigma(\bX)} \right| 
\\ 
= \dfrac{1}{\sqrt{2\pi\hat{\sigma}^2(\bX)}} \exp\left(-\dfrac{(y-\hat{f}(\bX))^2}{2\hat{\sigma}^2(\bX)}\right) \left|\dfrac{\sigma(\bX)-\hat{\sigma}(\bX)}{\sigma(\bX)} \right|\enspace .
\end{multline}
This upper-bound consists of two parts. One part which is the density of a Gaussian random variable (whose integral \emph{w.r.t.} $y$ is $1$) and a second term which is independent of $y$. Observe that this second term $|\sigma(\bX)-\hat{\sigma}(\bX)|$ is of the same order as $|\sigma^2(\bX)-\hat{\sigma}^2(\bX)|$. 
Indeed, notice that when $\sigma(\bX) > \hat{\sigma}(\bX)$
$$
\sigma^2(\bX)-\hat{\sigma}^2(\bX) = (\sigma(\bX)-\hat{\sigma}(\bX)) ( \sigma(\bX) + \hat{\sigma}(\bX)) \geq \left(  \sigma_0 + \dfrac{1}{\sqrt{s}} \right) (\sigma(\bX)-\hat{\sigma}(\bX))\enspace,
$$
where in the last inequality, we use Assumption~\ref{ass:assSigma} and the fact that $\hat{\sigma}(X) \geq 1/\sqrt{s}$.
Written differently, this means that $$ | \sigma(\bX)-\hat{\sigma}(\bX) |  \leq  \dfrac{\sqrt{s}}{1 +  \sigma_0 \sqrt{s} }  | \sigma^2(\bX)-\hat{\sigma}^2(\bX) | \leq \dfrac{ \sigma_0 \sqrt{s}}{1 +  \sigma_0 \sqrt{s} }  \dfrac{| \sigma^2(\bX)-\hat{\sigma}^2(\bX) |  } {\sigma_0}  \leq C | \sigma^2(\bX)-\hat{\sigma}^2(\bX) |\enspace,  $$
since $1/\sigma_0 \leq C$. The same reasoning holds in the case where $\sigma(\bX) < \hat{\sigma}(\bX)$ and then we conclude that 
$$
| \sigma(\bX)-\hat{\sigma}(\bX) |  \leq C | \sigma^2(\bX)-\hat{\sigma}^2(\bX) |\enspace.
$$
Injecting this bound into~\eqref{eq:eq1bis} and using again Assumption~\ref{ass:assSigma}, we deduce that
\begin{multline}
\label{eq:eq1}
\int_{\mathbb{R}}\left|\dfrac{1}{\sqrt{2\pi\hat{\sigma}^2(\bX)}} \exp\left(-\dfrac{(y-\hat{f}(\bX))^2}{2\hat{\sigma}^2(\bX)}\right)   -  \dfrac{1}{\sqrt{2\pi{\sigma^2(\bX)}}} \exp\left(-\dfrac{(y-\hat{f}(\bX))^2}{2\hat{\sigma}^2(\bX)}\right)\right|dy\\
\leq C  \left|\sigma(\bX)-\hat{\sigma}(\bX)\right|
\leq C 
\left|\sigma^2(\bX)-\hat{\sigma}^2(\bX)\right| \enspace .
\end{multline}
Let us now consider the second term in the decomposition of $\left|\hat{p}(y|\bX)-p(y|\bX)\right|$.
Since $x \mapsto \exp(-x)$ is $1$-Lipschitz on $\mathbb{R}_{+}$, from Assumption~\ref{ass:assSigma} we have that in the case where $(y-\hat{f}(\bX))^2 \geq  (y-f^*(\bX))^2$
\begin{multline*}
\left|\dfrac{1}{\sqrt{2\pi{\sigma^2(\bX)}}} \exp\left(-\dfrac{(y-\hat{f}(\bX))^2}{2\hat{\sigma}^2(\bX)}\right) -  \dfrac{1}{\sqrt{2\pi{\sigma^2(\bX)}}} \exp\left(-\dfrac{(y-f^*(\bX))^2}{2\hat{\sigma}^2(\bX)}\right)\right| \\
 =
\dfrac{1}{\sqrt{2\pi{\sigma^2(\bX)}}}
\exp\left(-\dfrac{(y-f^*(\bX))^2}{2{\hat{\sigma}}^2(\bX)}\right) 
\left| \exp\left( -\left( \dfrac{(y-\hat{f}(\bX))^2}{2\hat{\sigma}^2(\bX)} - \dfrac{(y-f^*(\bX))^2}{2{\hat{\sigma}}^2(\bX)} \right)\right) -1 \right|
\\
\leq \dfrac{1}{\sqrt{2\pi{\sigma}^2(\bX)} \times 2\hat{\sigma}^2(\bX)} \exp\left(-\dfrac{(y-f^*(\bX))^2}{2\hat{\sigma}^2(\bX)}\right)\left|(y-\hat{f}(\bX))^2- (y-f^*(\bX))^2\right|
\\
\leq \dfrac{C}{\sqrt{2\pi\hat{\sigma}^2(\bX)}\hat{\sigma}(\bX)} \exp\left(-\dfrac{(y-f^*(\bX))^2}{2\hat{\sigma}^2(\bX)}\right)\left|(y-\hat{f}(\bX))^2- (y-f^*(\bX))^2\right| \enspace.
\end{multline*}
Using the following decomposition
$$(y-\hat{f}(\bX))^2- (y-f^*(\bX))^2 = (\hat{f}(\bX)-f^*(\bX))^2 +2(y-f^*(\bX)(f^*(\bX)-\hat{f}(\bX))\enspace,$$
we deduce
\begin{multline}
\label{eq:eq2bis}
\left|\dfrac{1}{\sqrt{2\pi{\sigma^2(\bX)}}} \exp\left(-\dfrac{(y-\hat{f}(\bX))^2}{2\hat{\sigma}^2(\bX)}\right) -  \dfrac{1}{\sqrt{2\pi{\sigma^2(\bX)}}} \exp\left(-\dfrac{(y-f^*(\bX))^2}{2{\hat{\sigma}}^2(\bX)}\right)\right| \\
\leq \dfrac{C}{\sqrt{2\pi\hat{\sigma}^2(\bX)}\hat{\sigma}(\bX)} \exp\left(-\dfrac{(y-f^*(\bX))^2}{2\hat{\sigma}^2(\bX)}\right)\left(\left(\hat{f}(\bX)-f^*(\bX)\right)^2 + |y- f^*(\bX)||\hat{f}(\bX)-f^*(\bX)|\right) \enspace.
\end{multline}
In the case where $(y-\hat{f}(\bX))^2 \leq (y-f^*(\bX))^2$, we obtain similar bound as in the above equation by switching the role of $\hat{f}$ by $f^*$. 
Notice that $\int_{\mathbb{R}}  \dfrac{1}{\sqrt{2\pi\hat{\sigma}^2(\bX)}}  \exp \left(-\dfrac{(y-f^*(\bX))^2}{2\hat{\sigma}^2(\bX)}\right) \times |y- f^*(\bX)|
dy $ is the expectation of the r.v. $| Y-f^*(\bX) |$ where $Y$ is Gaussian with expectation $f^*(\bX)$ and variance $\hat{\sigma}^2(\bX)$.
Therefore, using the fact that
$\mathbb{E}\left[\left|Z-\mathbb{E}\left[Z\right] \right|\right] \leq \sqrt{{\rm Var}(Z)}$ for any real valued random variable $Z$, we get
\begin{multline*}
\int_{\mathbb{R}} \left|\dfrac{1}{\sqrt{2\pi{\sigma^2(\bX)}}} \exp\left(-\dfrac{(y-\hat{f}(\bX))^2}{2\hat{\sigma}^2(\bX)}\right) -  \dfrac{1}{\sqrt{2\pi{\sigma^2(\bX)}}} \exp\left(-\dfrac{(y-f^*(\bX))^2}{2{\hat{\sigma}}^2(\bX)}\right)\right| dy \\
\leq \dfrac{C}{\hat{\sigma}(\bX)} \left(\left(\hat{f}(\bX)-f^*(\bX)\right)^2 + \hat{\sigma}(\bX)|\hat{f}(\bX)-f^*(\bX)|\right) \enspace.
\end{multline*}
Finally, using that $\hat{\sigma}(\bX) \geq 1/\sqrt{s}$, we deduce
\begin{multline}
\label{eq:eq2}
\int_{\mathbb{R}}\left|\dfrac{1}{\sqrt{2\pi{\sigma^2(\bX)}}} \exp\left(-\dfrac{(y-\hat{f}(\bX))^2}{2\hat{\sigma}^2(\bX)}\right) -  \dfrac{1}{\sqrt{2\pi{\sigma^2(\bX)}}} \exp\left(-\dfrac{(y-f^*(\bX))^2}{2{\hat{\sigma}}^2(\bX)}\right)\right| dy \\
\leq C \left(\sqrt{s}\left(\hat{f}(\bX)-f^*(\bX)\right)^2+|\hat{f}(\bX)-f^*(\bX)|\right)\enspace.
\end{multline}
The remaining term in the decomposition of $\left|\hat{p}(y|\bX)-p(y|\bX)\right|$ is
\begin{multline*}
\left|\dfrac{1}{\sqrt{2\pi{\sigma^2(\bX)}}} \exp\left(-\dfrac{(y-{f}^*(\bX))^2}{2\hat{\sigma}^2(\bX)}\right) -  \dfrac{1}{\sqrt{2\pi{\sigma^2(\bX)}}} \exp\left(-\dfrac{(y-f^*(\bX))^2}{2{\sigma}^2(\bX)}\right)\right| \\ =
\dfrac{1}{\sqrt{2\pi{\sigma^2(\bX)}}}
\exp\left(-\dfrac{(y-f^*(\bX))^2}{2{\sigma}^2(\bX)}\right) \left| \exp\left( -\left( \dfrac{(y-{f}^*(\bX))^2}{2\hat{\sigma}^2(\bX)} - \dfrac{(y-f^*(\bX))^2}{2{\sigma}^2(\bX)} \right)\right)-1 \right|\enspace .
\end{multline*}
Hence, if $\sigma^2(\bX) \geq \hat{\sigma}^2(\bX)$, since $x \mapsto \exp(-x)$ is $1$-Lipschitz on $\mathbb{R}_{+}$,
we deduce from the above inequality that
\begin{multline*}
\left|\dfrac{1}{\sqrt{2\pi{\sigma^2(\bX)}}} \exp\left(-\dfrac{(y-{f}^*(\bX))^2}{2\hat{\sigma}^2(\bX)}\right) -  \dfrac{1}{\sqrt{2\pi{\sigma^2(\bX)}}} \exp\left(-\dfrac{(y-f^*(\bX))^2}{2{\sigma}^2(\bX)}\right)\right| \\ \leq
\dfrac{1}{\sqrt{2\pi{\sigma^2(\bX)}}}
\exp\left(-\dfrac{(y-f^*(\bX))^2}{2{\sigma}^2(\bX)}\right)\left|\dfrac{(y-{f}^*(\bX))^2}{2\hat{\sigma}^2(\bX)} - \dfrac{(y-f^*(\bX))^2}{2{\sigma}^2(\bX)} \right| \enspace .
\end{multline*}
Therefore, from Assumption~\ref{ass:assSigma}, and since $\hat{\sigma}^2(\bX) \geq 1/s$, we get in the case where  $\sigma^2(\bX) \geq \hat{\sigma}^2(\bX)$
\begin{multline}
\label{eq:eqAux1}
\left|\dfrac{1}{\sqrt{2\pi{\sigma^2(\bX)}}} \exp\left(-\dfrac{(y-{f}^*(\bX))^2}{2\hat{\sigma}^2(\bX)}\right) -  \dfrac{1}{\sqrt{2\pi{\sigma^2(\bX)}}} \exp\left(-\dfrac{(y-f^*(\bX))^2}{2{\sigma}^2(\bX)}\right)\right| \\ \leq   \dfrac{1}{\sqrt{2\pi{\sigma^2(\bX)}}}
\exp\left(-\dfrac{(y-f^*(\bX))^2}{2{\sigma}^2(\bX)}\right) Cs(y-f^*(\bX))^2\left|\hat{\sigma}^2(\bX)-\sigma^2(\bX)\right| \enspace.
\end{multline}
In the case where $\sigma^2(\bX) \leq \hat{\sigma}^2(\bX)$, using same arguments and additionally the fact that $\hat{\sigma}^2(\bX) \leq s$, we can obtain
\begin{multline}
\label{eq:eqAux2}
\left|\dfrac{1}{\sqrt{2\pi{\sigma^2(\bX)}}} \exp\left(-\dfrac{(y-{f}^*(\bX))^2}{2\hat{\sigma}^2(\bX)}\right) -  \dfrac{1}{\sqrt{2\pi{\sigma^2(\bX)}}} \exp\left(-\dfrac{(y-f^*(\bX))^2}{2{\sigma}^2(\bX)}\right)\right| \\ \leq   \dfrac{\sqrt{s}}{\sigma_0\sqrt{2\pi{\hat{\sigma}^2(\bX)}}}
\exp\left(-\dfrac{(y-f^*(\bX))^2}{2{\hat{\sigma}}^2(\bX)}\right) Cs(y-f^*(\bX))^2\left|\hat{\sigma}^2(\bX)-\sigma^2(\bX)\right|\enspace .
\end{multline}
Therefore, from Equation~\eqref{eq:eqAux1}, and~\eqref{eq:eqAux2}, we get
\begin{multline}
\label{eq:eq3}
\int_{\mathbb{R}} \left|\dfrac{1}{\sqrt{2\pi{\sigma^2(\bX)}}} \exp\left(-\dfrac{(y-{f}^*(\bX))^2}{2\hat{\sigma}^2(\bX)}\right) -  \dfrac{1}{\sqrt{2\pi{\sigma^2(\bX)}}} \exp\left(-\dfrac{(y-f^*(\bX))^2}{2{\sigma}^2(\bX)}\right)\right| dy \\ \leq C s^{5/2} \left|\hat{\sigma}^2(\bX)-\sigma^2(\bX)\right|\enspace,
\end{multline}
where we used the fact that the integral \wrt $y$ is the variance of Gaussian r.v. with variance $\hat{\sigma}^2(\bX)$ and is then such that
 $ \int_{\mathbb{R}} \dfrac{1}{\sqrt{2\pi{\hat{\sigma}^2(\bX)}}}
\exp\left(-\dfrac{(y-f^*(\bX))^2}{2{\hat{\sigma}}^2(\bX)}\right) (y-f^*(\bX))^2  dy = \hat{\sigma}^2(\bX)  \leq s$.
The combination of Equations~\eqref{eq:eq1},~\eqref{eq:eq2}, and~\eqref{eq:eq3} yields the result.
\end{proof}

Now, we provide the proof of Theorem~\ref{thm:consistance}.

\begin{proof}[Proof of Theorem~\ref{thm:consistance}]

We prove the consistency $\hat{\Gamma}$ {\it w.r.t.}
the symmetric difference distance $\mathcal{H}$.
We have that
\begin{equation}
\label{eq:eqDecompHamming0}
\mathcal{H}(\hat{\Gamma}) \leq \mathbb{E}\left[\int_{\hat{\Gamma}(\bX,\zeta) \triangle \bar{\Gamma}(\bX,\zeta) } dy \right] + \mathbb{E}\left[\int_{\bar{\Gamma}(\bX,\zeta) \triangle \Gamma^*(\bX)} dy\right] \enspace .    
\end{equation}
We bound the first term in the {\it r.h.s.} in 
the above inequality.
\begin{multline*}
\mathbb{E}\left[\int_{\hat{\Gamma}(\bX,\zeta) \triangle \bar{\Gamma}(\bX,\zeta) } dy \right]
 = \mathbb{E}\left[\int_{\mathbb{R}} \left|\one_{\{y \in \hat{\Gamma}(\bX,\zeta)\}}-\one_{\{y \in \bar{\Gamma}(\bX,\zeta)\}}\right|\right]
\\= \mathbb{E}\left[\int_{\mathbb{R}}\bigg|\one_{\{\hat{G}(\hat{p}(y|\bX,\zeta)) \leq \ell\}}-\one_{\{\bar{G}(\hat{p}(y|\bX,\zeta)) \leq \ell\}}\bigg|dy\right]\enspace .
\end{multline*}
Therefore, from Equations~\eqref{eq:decompSize} and~\eqref{eq:sizeProbRate}, we deduce
\begin{equation}
\label{eq:eqdecompHamming1}
\mathbb{E}\left[\int_{\hat{\Gamma}(\bX,\zeta) \triangle \bar{\Gamma}(\bX,\zeta) } dy \right] \leq  \dfrac{Cs}{\sqrt{N}}\enspace.
\end{equation}
Now, we study the second term in the {\it r.h.s.}
of Equation~\eqref{eq:eqDecompHamming0}.
We observe that if $y \in \bar{\Gamma}(\bX, \zeta) \setminus \Gamma_{\ell}^*(\bX)$ the following holds
\begin{itemize}
    \item on the event $\{\bar{G}^{-1}(\ell) \geq G^{-1}(\ell)\}$, $\left|\hat{p}(y|\bX,\zeta) -p(y|\bX)\right|\geq \left|p(y|\bX)-G^{-1}(\ell) \right| ,$ 
    \item on the event $\{\bar{G}^{-1}(\ell) < G^{-1}(\ell)\}$, 
    \begin{equation*}
    {\rm either} \;\; \hat{p}(y|\bX,\zeta) \in (\bar{G}^{-1}(\ell), G^{-1}(\ell)) \;\;
    {\rm or} \;\; \left|\hat{p}(y|\bX,\zeta) -p(y|\bX)\right|\geq \left|p(y|\bX)-G^{-1}(\ell) \right|\enspace. 
\end{equation*}    
\end{itemize}
Note that similar reasoning holds if $y \in  \Gamma_{\ell}^*(\bX) \setminus \bar{\Gamma}(\bX, \zeta)$. 
Therefore, we deduce that conditional on $\mathcal{D}_n$,
\begin{multline*}
\mathbb{E}\left[\int_{\bar{\Gamma}(\bX,\zeta) \triangle \Gamma^*(\bX)} dy\right] \leq 
\mathbb{E}\left[\int_{\mathbb{R}}\one_{\{\left|\hat{p}(y|\bX,\zeta) -p(y|\bX)\right|\geq  \left|p(y|\bX)-G^{-1}(\ell) \right|\}} dy \right] \\+ \one_{\{\bar{G}^{-1}(\ell) < G^{-1}(\ell)\}} \mathbb{E}\left[\int_{\mathbb{R}} \one_{\{\hat{p}(y|\bX,\zeta) \in (\bar{G}^{-1}(\ell), G^{-1}(\ell))\}} dy\right] \\+ 
\one_{\{\bar{G}^{-1}(\ell) \geq G^{-1}(\ell)\}} \mathbb{E}\left[\int_{\mathbb{R}} \one_{\{\hat{p}(y|\bX,\zeta) \in ({G}^{-1}(\ell), \bar{G}^{-1}(\ell))\}} dy\right]\enspace.
\end{multline*}
Using first the definition of $\bar{G}$ and then the fact that $\bar{G}(\bar{G}^{-1}(\ell)) = G(G^{-1}(\ell)) = \ell $ in this last inequality, we deduce the following
\begin{multline*}
\mathbb{E}\left[\int_{\bar{\Gamma}(\bX,\zeta) \triangle \Gamma^*(\bX)} dy\right] \leq 
\mathbb{E}\left[\int_{\mathbb{R}}\one_{\{\left|\hat{p}(y|\bX,\zeta) -p(y|\bX)\right|\geq  \left|p(y|\bX)-G^{-1}(\ell) \right|\}} dy \right] \\+
\mathbb{E}\left[\left|G(G^{-1}(\ell)) - \bar{G}(G^{-1}(\ell)) \right|\right]\enspace.
\end{multline*}
Now, we observe that
\begin{multline*}
\mathbb{E}\left[\left|G(G^{-1}(\ell)) - \bar{G}(G^{-1}(\ell)) \right|\right] \leq \mathbb{E}\left[\int_{\mathbb{R}} \left|\one_{\{p(y|\bX) \geq G^{-1}(\ell)\}} - \one_{\{\hat{p}(y|\bX, \zeta) \geq G^{-1}(\ell)\}}\right| dy\right]\\
\leq \mathbb{E}\left[\int_{\mathbb{R}} \one_{\{\left|\hat{p}(y|\bX,\zeta) -p(y|\bX)\right|\geq  \left|p(y|\bX)-G^{-1}(\ell) \right|\}} dy\right]\enspace.
\end{multline*}
Therefore, we have obtained
\begin{eqnarray}
\label{eq:eqDecompHamming2}
\mathbb{E}\left[\int_{\bar{\Gamma}(\bX,\zeta) \triangle \Gamma^*(\bX)} dy\right] &\leq& 2
\mathbb{E}\left[\int_{\mathbb{R}}\one_{\{\left|\hat{p}(y|\bX,\zeta) -p(y|\bX)\right|\geq  \left|p(y|\bX)-G^{-1}(\ell) \right|\}} dy \right] \\ \nonumber
& \leq & 2 \int_{\mathbb{R}} \mathbb{P}\left(\left|\hat{p}(y|\bX,\zeta) -p(y|\bX)\right|\geq  \left|p(y|\bX)-G^{-1}(\ell) \right|\right) dy \enspace.
\end{eqnarray}
Let us consider the term in the {\it r.h.s} of Equation~\eqref{eq:eqDecompHamming2}. 
Let $\delta > 0$, we have that
\begin{multline*}
\int_{\mathbb{R}} \mathbb{P}\left(\left|\hat{p}(y|\bX,\zeta) -p(y|\bX)\right|\geq  \left|p(y|\bX)-G^{-1}(\ell) \right|\right) dy \leq \int_{\mathbb{R}}\mathbb{P}\left(\left|\hat{p}(y|\bX,\zeta) -p(y|\bX)\right|\geq \delta\right) dy\\ + \int_{\mathbb{R}}  \mathbb{P}\left(\left|p(y|\bX)-G^{-1}(\ell)\right| \leq  \delta\right)  dy \enspace .  
\end{multline*}
From Markov's inequality, we deduce
\begin{multline}
\label{eq:eqDecompHamming3}
\int_{\mathbb{R}} \mathbb{P}\left(\left|\hat{p}(y|\bX,\zeta) -p(y|\bX)\right|\geq  \left|p(y|\bX)-G^{-1}(\ell) \right|\right) dy \leq   \dfrac{1}{\delta} \mathbb{E}\left[\int_{\mathbb{R}}\left|\hat{p}(y|\bX,\zeta) -p(y|\bX)\right|dy\right] \\+ G(G^{-1}(\ell)-\delta)-G(G^{-1}(\ell)+\delta) \enspace.
\end{multline}
Since $\hat{p}$ is supported on $[-s,s]$, we observe that
\begin{multline}
\label{eq:decompL1Risk}
\mathbb{E}\left[\int_{\mathbb{R}}\left|\hat{p}(y|\bX,\zeta)-p(y|\bX)\right|{\rm d}y\right] 
  = 
\mathbb{E}\left[\int_{[-s,s]}\left|\hat{p}(y|\bX, \zeta)-p(y|\bX)\right|{\rm d}y\right] +
\mathbb{E}\left[\int_{|y| \geq s}^{+\infty}p(y|\bX) {\rm d}y\right]\\
 \leq  \mathbb{E}\left[\int_{[-s,s]}\left|\hat{p}(y|\bX)-p(y|\bX)\right|{\rm d}y\right]+ 2su + \mathbb{E}\left[\int_{|y| \geq s}^{+\infty}p(y|\bX) {\rm d}y\right]\enspace.
\end{multline}
Now, we observe that
\begin{equation*}
\mathbb{E}\left[\int_{s}^{+\infty}p(y|\bX) {\rm d}y\right]
= \mathbb{E}\left[\one_{\{|f^*(X)|\leq \frac{s}{2}\}}\int_{s}^{+\infty}p(y|\bX) {\rm d}y\right]+
\mathbb{E}\left[\one_{\{|f^*(X)|> \frac{s}{2}\}}\int_{s}^{+\infty}p(y|\bX) {\rm d}y\right]\enspace.
\end{equation*}
From Markov inequality and Assumption~\ref{ass:assFstar}, the second term of the r.h.s. in the above inequality is bounded by
\begin{equation*}
\mathbb{E}\left[\one_{\{|f^*(\bX)|> \frac{s}{2}\}}\int_{s}^{+\infty}p(y|\bX) {\rm d}y\right] \leq \mathbb{E}\left[\one_{\{|f^*(\bX)|> \frac{s}{2}\}}\right] \leq
\dfrac{2C_1}{s} \enspace.
\end{equation*}
On the other hand, we observe that
\begin{multline*}
\mathbb{E}\left[\one_{\{|f^*(\bX)|\leq \frac{s}{2}\}}\int_{s}^{+\infty}p(y|\bX) {\rm d}y\right]
\leq \mathbb{E}\left[\int_s^{+\infty} \dfrac{1}{\sqrt{2\pi\sigma^2(\bX)}} \exp\left(-\dfrac{(y-s/2)^2}{2\sigma^2(\bX)}\right) {\rm d}y\right] \\
= \mathbb{E}\left[\int_{s/2}^{+\infty} \dfrac{1}{\sqrt{2\pi\sigma^2(\bX)}} \exp\left(-\dfrac{y^2}{2\sigma^2(\bX)}\right) {\rm d}y\right]\enspace.
\end{multline*}
Therefore, Assumption~\ref{ass:assSigma} and standard result on Gaussian tails yields for $s \geq 1$
\begin{equation*}
\mathbb{E}\left[\one_{\{|f^*(\bX)|\leq \frac{s}{2}\}}\int_{s}^{+\infty}p(y|\bX) {\rm d}y\right]
\leq \dfrac{1}{\sqrt{2\pi}}\exp\left(- \dfrac{s^2}{8\sigma_1^2}\right)\enspace.
\end{equation*}
Hence combining the above inequalities, we get for $s \geq 1$
\begin{equation*}
\mathbb{E}\left[\int_{s}^{+\infty}p(y|\bX) {\rm d}y\right] \leq C'\left[\exp\left(-\frac{s^2}{8\sigma_1^2}\right) + \frac{C}{s}\right]\enspace,  
\end{equation*}
where $C$ and $C'$ are two positive constants. Note that similar arguments yields
\begin{equation*}
 \mathbb{E}\left[\int_{-\infty}^{-s}p(y|\bX) {\rm d}y\right] \leq C'\left[\exp\left(-\frac{s^2}{8\sigma_1^2}\right) + \frac{C}{s}\right]\enspace. 
\end{equation*}
Therefore considering Equation~\eqref{eq:decompL1Risk} and 
Proposition~\ref{prop:propEstim} and defining $s = \log(\min(n,N))$, we get 
\begin{equation*}
\lim_n    \dfrac{1}{\delta} \mathbb{E}\left[\int_{\mathbb{R}}\left|\hat{p}(y|\bX,\zeta) -p(y|\bX)\right|dy\right] = 0 \enspace.
\end{equation*}    
Hence we obtain from Equation~\eqref{eq:eqDecompHamming3} that for all $\delta > 0$
\begin{equation*}
\limsup_n \int_{\mathbb{R}} \mathbb{P}\left(\left|\hat{p}(y|\bX,\zeta) -p(y|\bX)\right|\geq  \left|p(y|\bX)-G^{-1}(\ell) \right|\right) dy \leq G(G^{-1}(\ell)-\delta)-G(G^{-1}(\ell)+\delta)\enspace.
\end{equation*}
Since $G$ is continuous, with $\delta \rightarrow 0$,
we get
\begin{equation*}
\lim_n \int_{\mathbb{R}} \mathbb{P}\left(\left|\hat{p}(y|\bX,\zeta) -p(y|\bX)\right|\geq  \left|p(y|\bX)-G^{-1}(\ell) \right|\right) dy  = 0 \enspace. 
\end{equation*}
The above equation together with  Equation~\eqref{eq:eqDecompHamming0}, ~\eqref{eq:eqdecompHamming1},~\eqref{eq:eqDecompHamming2} yields the desired result.
\end{proof}

\subsection{Rates of convergence}
We start this section with a result on the estimation error of $\hat{p}$
\wrt the sup-norm. 
\begin{proposition}  
\label{prop:supNormEstimator} 
Let $s = \log(\min(n,N))$.
Under Assumptions~\ref{ass:assSigma},~\ref{ass:regularity}, and \ref{ass:StrongDensityAssumption}, we have that
\begin{multline*}
\sup_{(\bx,y) \in \mathcal{C} \times [-s,s]} \left|\hat{p}(y|\bx)-p(y|\bx)\right|
\leq \\ C\left(s\sup_{\bx \in \mathcal{C}}\left(\hat{f}(\bx)-f^*(\bx)\right)^2+ s^{2} \sup_{\bx \in \mathcal{C}}\left|\hat{f}(\bx)-f^*(\bx)\right|+ s^3\sup_{\bx \in \mathcal{C}}\left|\hat{\sigma}^2(\bx) -\sigma^2(\bx)\right|\right) \enspace,
\end{multline*}
where $C > 0$ is a constant which depends on $f^*$, $\sigma^2$, and on the set $\mathcal{C}$.
\end{proposition}
\begin{proof}
We consider the same decomposition into $3$ that we used in the proof of Proposition~\ref{prop:propEstim}. 
Using the fact that 
$\hat{\sigma}(\bx) \geq \frac{1}{\sqrt{s}}$ and Assumption~\ref{ass:assSigma}, we get for all $\bx \in \mathcal{C}$, and $y\in [-s,s]$ (\emph{c.f.,} Eq.~\eqref{eq:eq1bis}), the first term is controlled as follows:
\begin{equation}
\label{eq:eqsupnorm1}
\left|\dfrac{1}{\sqrt{2\pi\hat{\sigma}^2(\bx)}} \exp\left(-\dfrac{(y-\hat{f}(\bx))^2}{2\hat{\sigma}^2(\bx)}\right)   -  \dfrac{1}{\sqrt{2\pi{\sigma^2(\bx)}}} \exp\left(-\dfrac{(y-\hat{f}(\bx))^2}{2\hat{\sigma}^2(\bx)}\right)\right| \leq 
C s \sup_{x \in \mathcal{C}} \left|\hat{\sigma}^2(\bx) -\sigma^2(\bx)\right| \enspace.
\end{equation}
According to the second term, from Assumptions~\ref{ass:regularity} and~\ref{ass:StrongDensityAssumption}, and since $f^*$ is a Lipschitz function on the compact $\mathcal{C}$, we have that $\left|f^*(\bx) \right|\leq s$ for $n,N$ large enough. Therefore,
using the fact that $x \mapsto \exp(-x)$ is 1-Lipschitz on $\mathbb{R}_{+}$ and that $\frac{1}{s} \leq \hat{\sigma}^2(\bx)$, we get (\emph{c.f.,} Eq.~\eqref{eq:eq2bis})
\begin{multline}
\label{eq:eqsupnorm2}
\left|\dfrac{1}{\sqrt{2\pi{\sigma^2(\bx)}}} \exp\left(-\dfrac{(y-\hat{f}(\bx))^2}{2\hat{\sigma}^2(\bx)}\right) -  \dfrac{1}{\sqrt{2\pi{\sigma^2(\bx)}}} \
\exp\left(-\dfrac{(y-f^*(\bx))^2}{2{\hat{\sigma}}^2(\bx)}\right)\right|  \leq\\
C\left(s\sup_{\bx \in \mathcal{C}}\left(\hat{f}(\bx)-f^*(\bx)\right)^2+ s^{2} \sup_{\bx \in \mathcal{C}}\left|\hat{f}(\bx)-f^*(\bx)\right| \right) \enspace .  
\end{multline}
Finally, considering the last term, we deduce from~\eqref{eq:eqAux1} and~\eqref{eq:eqAux2} that
\begin{equation}
\label{eq:eqsupnorm3}
\left|\dfrac{1}{\sqrt{2\pi{\sigma^2(\bx)}}} \exp\left(-\dfrac{(y-{f}^*(\bx))^2}{2\hat{\sigma}^2(\bx)}\right) -  \dfrac{1}{\sqrt{2\pi{\sigma^2(\bx)}}} \exp\left(-\dfrac{(y-f^*(\bx))^2}{2{\sigma}^2(\bx)}\right)\right|\leq Cs^3 \sup_{\bx \in \mathcal{C}}\left|\hat{\sigma}^2(\bx) -\sigma^2(\bx)\right|\enspace.
\end{equation}
The combination of Equations~\eqref{eq:eqsupnorm1},~\eqref{eq:eqsupnorm2}, and~\eqref{eq:eqsupnorm3} gives the proposition.
\end{proof}

\begin{proof}[Proof of Proposition~\ref{prop:excessRiskUnderMargin}]
We recall that
\begin{equation*}
\mathcal{E}_{\ell}\left(\bar{\Gamma}\right) =\mathbb{E}\left[\int_{\bar{\Gamma}(\bX, \zeta) \triangle \Gamma^*_{\ell}(\bX)} \left|p(y|\bX)-\lambda^*_{\ell} \right| dy\right]\enspace. \end{equation*}
Now, we observe that for $y \in \bar{\Gamma}(\bX, \zeta) \triangle \Gamma^*_{\ell}(\bX)$
\begin{equation*}
\left|p(y|\bX)-\lambda^*_{\ell} \right| \leq \left|\hat{p}(y|\bX,\zeta)- p(y|\bX)\ \right| + \left|\bar{\lambda}_{\ell} - \lambda^*_{\ell}\right|\enspace,
\end{equation*}
where we recall that $\bar{\lambda}_{\ell} := \bar{G}^{-1}(\ell)$, with $\bar{G}$ defined in Eq.~\eqref{eq:pseudoOracle}.
Using similar arguments as those used in the proof of Theorem 4.4 in~\cite{Denis_Hebiri_Zaoui20} that is inspired by Theorem 2.12 in~\cite{Bobkov_Ledoux14}, 
it is not difficult to see that conditional on $\mathcal{D}_n$\enspace,
\begin{equation*}
 \left|\bar{\lambda}_{\ell} - \lambda^*_{\ell}\right| \leq \sup_{(\bx,y) \in \mathcal{C} \times \mathbb{R}} \left|\hat{p}(y|\bx)-p(y|\bx) \right| + u := \hat{m}(u)\enspace.
\end{equation*}
Therefore, we deduce that
\begin{equation*}
\mathbb{E}\left[\int_{\bar{\Gamma}(\bX, \zeta) \triangle \Gamma^*_{\ell}(\bX)} \left|p(y|\bX)-\lambda^*_{\ell} \right| dy \right]
\leq 2\hat{m}(u) \mathbb{E}\left[\int_{\mathbb{R}} \one_{\{\left|p(y|\bX)-\lambda^*_{\ell} \right| \leq 2\hat{m}(u)\}} {\rm d}y\right] \enspace .
\end{equation*}
Hence from the above inequality, and Assumption~\ref{ass:marginAss} we get
\begin{equation*}
\mathbb{E}\left[\mathcal{E}_{\ell}\left(\bar{\Gamma}\right)\right] \leq 2^{1+\alpha}c_0\mathbb{E}\left[\hat{m}(u)^{1+\alpha}\right] \enspace .   
\end{equation*}
Therefore, from Equations~\eqref{eq:excessRiskFirstdecomp} and~\eqref{eq:diffRiskR}, we obtain the following with $s  = \log(\min(n,N))$
\begin{equation*}
\mathbb{E}\left[\mathcal{E}_\ell \left(\hat{\Gamma}\right)\right] \leq C\left( \mathbb{E}\left[\left(\sup_{(x,y) \in \mathbb{R} \times \mathcal{C}} \left|\hat{p}(y|\bx)-p(y|\bx) \right|\right)^{1+\alpha}\right] + u^{1+\alpha} + \dfrac{\log(N)}{N}\right) \enspace
.    
\end{equation*}
Finally, since $\hat{p}$ is supported on $[-s,s]$, we have
\begin{equation*}
\sup_{(\bx,y) \in \mathcal{C} \times \mathbb{R}} \left|\hat{p}(y|\bx)-p(y|\bx) \right|  \leq  \sup_{(\bx,y) \in \mathcal{C} \times [-s,s]} \left|\hat{p}(y|\bx)-p(y|\bx) \right|
+ \sup_{(\bx,y) \in \mathcal{C} \times \mathbb{R}\setminus[-s,s]} p(y|\bx) \enspace .
\end{equation*}
For $n,N$ large enough, we can assume, since $f^*$ is bounded,  that $\left|f^*(\bX)\right| \leq s/2$. From Assumption~\ref{ass:assSigma}, we have for $n,N$  large enough 
\begin{equation*}
 \sup_{(\bx,y) \in \mathcal{C} \times \mathbb{R}\setminus[-s,s]} p(y|\bx) \leq C\exp\left(-\frac{s^2}{8\sigma_1^2}\right) \leq \exp(-s) \leq \frac{C}{\min(n,N)}\enspace, 
 \end{equation*}
which yields the desired result.

\end{proof}

\begin{proof}[Proof of Theorem~\ref{thm:RatesOfcve}]

The proof is a straightforward application of Proposition~\ref{prop:excessRiskUnderMargin}, ~\ref{prop:supNormEstimator}, and Theorem~\ref{thm:cveKnn}, where we also use the fact that for $n,N$ large enough
\begin{equation*}
\left|\hat{\sigma}^2(\bx)-\sigma^2(\bx) \right| \leq   \left|\tilde{\sigma}^2(\bx)-\sigma^2(\bx) \right| \enspace. 
\end{equation*}
\end{proof}

\end{document}